\documentclass[12pt]{amsart}

\usepackage{latexsym,color,amsmath,amsthm,amssymb,amscd,amsfonts}
\usepackage{tikz}
\usetikzlibrary{arrows}
\usepackage{scalefnt}
\usepackage[font=small,labelfont=bf]{caption}
\usepackage{float}
\usepackage{graphicx}
\usepackage{subcaption}
\usepackage{relsize}
\usepackage{amsopn}
\usepackage{mathrsfs}  %
\usepackage[hidelinks]{hyperref}
\usepackage{bm}
\usepackage{bbm}
\usepackage[shortlabels]{enumitem}
\usepackage{multicol}
\usepackage{float} 
\usepackage{todonotes}
\usepackage{cite}
\usepackage{soul}

\usepackage[margin=2.5cm]{geometry}

\graphicspath{ {images/} }
\newtheoremstyle{nonum}{}{}{\itshape}{}{\bfseries}{.}{ }{\thmnote{#3}}

\newtheorem{thm}{Theorem}[section]
\newtheorem*{thm*}{Theorem}

\newtheorem{cor}[thm]{Corollary}
\newtheorem{lem}[thm]{Lemma}
\newtheorem{prop}[thm]{Proposition}

\newtheorem{exm}[thm]{Example}
\newtheorem{definition}[thm]{Definition}
\newtheorem*{definition*}{Definition}

\newtheorem{manualtheoreminner}{Theorem}
\newenvironment{manualtheorem}[1]{%
	\renewcommand\themanualtheoreminner{#1}%
	\manualtheoreminner
}{\endmanualtheoreminner}

\newtheorem{manuallemmainner}{Lemma}
\newenvironment{manuallemma}[1]{%
	\renewcommand\themanuallemmainner{#1}%
	\manuallemmainner
}{\endmanualtheoreminner}

\newcommand{\innerthmname}{}%

\theoremstyle{definition}

\theoremstyle{nonum}

\theoremstyle{remark}

\newtheorem*{rems*}{Remarks}
\newtheorem{rems}[thm]{Remarks}
\newtheorem{rem}[thm]{Remark}

\newcommand{\R}{\mathbb R}
\newcommand{\RR}{\mathbb R}

\def\A{{\mathcal A}}
\def\L{{\mathcal L}}

\newcommand{\dual}{^\circ}
\newcommand{\ind}{\mathbbm{1}}

\def\sp{{\rm sp }}

\newcommand{\iprod}[2]{\langle #1,#2 \rangle} %

\def\L{{\calL}}

\def\calP{{\mathcal P}}

\def\grad{{\nabla}}

\def\pa{\partial}

\def\dom{{\rm dom}\,}
\def\cvx{{\rm Cvx}}

\def\eps{{\varepsilon}}
\def\dom{{\rm dom}}
\def\conv{{\rm conv}}

\def\A{{\mathcal A}}

\def\R{\mathbb R}

\def\cvx{{\rm Cvx}}

\def\L{{\mathcal L}}

\def \sp#1{\langle #1\rangle}

\def\v{\varphi}

\def\iff{\;\;\Leftrightarrow\;\;}

\def\dom{\text{dom}}
\def \cgrad{\partial^c}

\def \grad{\partial}

\makeatletter
\def\moverlay{\mathpalette\mov@rlay}
\def\mov@rlay#1#2{\leavevmode\vtop{%
		\baselineskip\z@skip \lineskiplimit-\maxdimen
		\ialign{\hfil$\m@th#1##$\hfil\cr#2\crcr}}}
\newcommand{\charfusion}[3][\mathord]{
	#1{\ifx#1\mathop\vphantom{#2}\fi
		\mathpalette\mov@rlay{#2\cr#3}
	}
	\ifx#1\mathop\expandafter\displaylimits\fi}
\makeatother

\begin{document}
\title{Optimal measure transportation with respect to non-traditional costs}
\author{S. Artstein-Avidan, S. Sadovsky, K. Wyczesany}
\date{}
\maketitle
\ 
\begin{abstract}
 We study optimal mass transport problems between two measures with respect to a non-traditional cost function, i.e.~a cost $c$ which can attain the value $+\infty$. %
 We define the notion of  $c$-compatibility and strong-$c$-compatibility of two measures, and 
 prove that   if there is a finite-cost plan between the measures then the measures must be $c$-compatible, and if in addition the two measures are strongly $c$-compatible, then there is an optimal plan  concentrated  on a $c$-subgradient of a $c$-class function. This function is the so-called potential of the plan. 

We give two proofs of this theorem, under slightly different assumptions. In the first we utilize the notion of $c$-path-boundedness, showing that strong $c$-compatibility  implies a strong connectivity result for a directed graph associated with an optimal map. Strong connectivity of the graph implies that the $c$-cyclic monotonicity of the support set (which follows from classical reasoning)  guarantees its $c$-path-boundedness, implying, in turn, the existence of a potential. We also give a constructive proof, in the case when one of the measures is discrete. This approach adopts a new notion of `Hall polytopes', which we introduce and study in depth, to which we apply a version of Brouwer's fixed point theorem to prove the existence of a potential in this case.

\end{abstract}

\section{Introduction and results}

The Monge transport problem is concerned with finding a \emph{transport map} moving mass from one probability measure\footnote{All considered measures are Borel measures on Polish spaces, which are
complete, separable metric spaces equipped with their Borel $\sigma$-algebra.} to another, in a way which is efficient with respect to some \emph{cost function}. The most widely studied case of this problem is for the quadratic cost $c(x,y)=\|x-y\|_2^2/2$, for which the Brenier--Gangbo-McCann theorem \cite{Brenier, gangbo-mccann} implies that under mild conditions on the measures involved, optimal transport maps exist and are given by gradients of convex functions. In this work the main emphasis will be on \emph{non-traditional} cost functions, i.e. costs that can attain the value $+\infty$, as this project is motivated by the study of transportation with respect to the so-called \emph{polar cost} given by
\begin{equation}\label{def:polar-cost}
    p(x,y)=-\ln(\sp{x,y}-1),
\end{equation}
where $p(x,y)=+\infty$ if $\iprod{x}{y}\le 1$. This cost function is linked with the polarity transform (see \cite{ArtsteinMilmanHidden, ArtsteinRubinstein}), similarly to the strong connection of the quadratic cost with the Legendre transform. Transportation with respect to polar cost was first considered in \cite{hila}. 

We provide necessary conditions on pairs of measures, together with a cost $c$, for which finite cost plans exist. To this end, we discuss the class of functions connected with a cost, called its $c$-class (see the definition in Section \ref{subsec-ctransform-cclass}). The optimality of a plan is linked with the possibility of finding a ``potential'' for the plan, which is a $c$-class function such that the plan lies on its $c$-subgradient (yet another important notion we discuss in depth, see the definition in Equation \eqref{eq:c-subgradient}).

We will see shortly  that the mere existence of a finite cost plan between two measures $\mu$ and $\nu$ implies that the two measures considered are $c$-compatible, namely that for any measurable set $A$ in the measure space $(X,\mu)$, one has that $\mu(A) \le \nu(\{y: \exists x\in A, \,\,c(x,y)<\infty\})$. This is quite intuitive -- all points (up to measure $0$) in $A$ must be mapped to points in the target space with which they have finite cost. This $c$-compatibility of two measures is thus a necessary condition (for the formal definition of $c$-compatibility see Definition \ref{def:c-comp-finite}, and for the statement of the necessity of this condition see Lemma \ref{lem:pre-strassen}). As an example we will show (see Example \ref{example:ccompatible-not-enough}) that $c$-compatibility is not a sufficient condition for the existence of a finite cost plan.  However,  if a finite cost plan exists,  a slight strengthening of $c$-compatibility condition in which we demand a strict inequality is already sufficient to ensure that the optimal plan has a potential.  We will show later why our notion of ``strong compatibility'' is a very natural strengthening of compatibility, and discuss cases where two measures are $c$-compatible but not strongly $c$-compatible and how this implies that the transport problem is decomposable into sub-problems. 

In this note we only consider symmetric cost functions $c:X\times X\to (-\infty,\infty]$ with $c(x,y) = c(y,x)$,  but to see the difference between the two variables we denote the second copy of $X$ by $Y$ and write $c:X \times Y \to (-\infty,\infty]$.  Our results hold for the non-symmetric case as well, with only minor adjustments. We will also add a lower-bound assumption on the cost $c$ which allows to integrate it and its marginals (see also Example \ref{exm:minfinity}).  We say that $c$ is {\em essentially bounded from below with respect to $\mu$ and $\nu$} if there exist functions $a(x)\in L^1(\mu)$, $b(x)\in L^1(\nu)$ such that $c(x,y)\ge a(x)+b(y)$. For the polar cost  this condition is satisfied if, for example, both measures have finite second moment.

Our main theorem is the following 
(here $\partial^c \varphi$ denotes the $c$-subgradient of $\varphi$, see the definition in equation \eqref{eq:c-subgradient}, and $\Pi(\mu, \nu)$ denotes all transport plans between $\mu$ and $\nu$, see the beginning of Section \ref{sec:background}).

\begin{thm}\label{thm:transport-polar-comp}
	Let $X=Y$ be a Polish space, let $c: X\times Y \to (-\infty, \infty]$ be a continuous and symmetric cost function, essentially bounded from below with respect to probability measures $\mu\in {\mathcal P}(X)$ and $\nu\in {\mathcal P}(Y)$. Assume $\mu$ and $\nu$  are strongly $c$-compatible, namely satisfy that for any measurable $A\subset X$ we have 
	\[ \mu(A) + \nu(\{y\in Y: \forall x\in A, \,\, c(x,y) = \infty \}) <1.\]   
	If there exists \emph{some} finite cost plan transporting $\mu$ to $\nu$, then there exists a  {$c$-class function} $\varphi$ and an optimal transport plan $\pi \in \Pi(\mu, \nu)$ concentrated on $\partial^c \varphi$.
\end{thm}

The proof uses results from \cite{paper1} on $c$-path-boundedness, which is a notion that replaces $c$-cyclic monotonicity from the Rockafellar-Rochet-R\"uschendorf result  (see	 \cite{rockafellar,rochet,Ruschendorf1996c}) in the case when the cost is non-traditional.  The $c$-path-boundedness is a necessary and sufficient condition for a set to be included in a $c$-subgradient of a $c$-class function.

Since the initial main interest for us in developing this theory concerned the polar cost, in which case we have a precise form for $c$-subgradients, let us state the relevant theorem, which is almost a direct application of the theorem above, together with some simple analysis of polar-subgradients as performed in \cite{ArtsteinRubinstein}.  By $\cvx_0(\RR^n)$ we denote a class of   lower semi-continuous convex functions from $\R^n$ to $[0,\infty]$ which take the value zero at the origin. By $\A$ we denote the polarity transform on the class $\cvx_0(\RR^n)$, defined in \cite{ArtsteinMilmanHidden} and given in \eqref{def:A-transform}.

	\begin{thm}\label{thm:transpolar}
	Let $X = Y = \RR^n$ and let $\mu,\, \nu\in {\mathcal P}(\RR^n)$ be probability measures with finite second moment, which are strongly $p$-compatible where $p(x,y)=-\ln(\iprod{x}{y}-1)_+$ is the polar cost, that is
	\[\mu(K) + \nu (K^\circ) < 1 \]
	for any convex set $K$ with $\mu(K)\neq 0,1$. Assume further that $\mu$ is absolutely continuous.
	Assume there exists some finite cost plan mapping $\mu$ to $\nu$. 
	Then there exists $\varphi \in \cvx_0(\RR^n)$ such that $\grad \dual \varphi$  is an optimal transport map between $\mu$ and $\nu$, where 
	\[ \grad \dual \varphi(x) = \{ y\in \R^n:\, \varphi(x) \A \varphi(y) = \iprod{x}{y}-1>0\}. \] 
	In particular, for $\mu$-almost every $x$,  the set $\grad \dual \varphi(x)$ is a singleton. 
\end{thm}

We remark that the existence of a potential function for the cost $p$ and other non-traditional costs leads naturally to the question regarding regularity of such potentials (as introduced by Caffarelli in \cite{Caffarelli1992regularity} and developed, among others, by Trudinger and Wang in  \cite{Trudinger2009strict}). In this work we do not pursue this direction, and instead focus on the analysis of the existence of potentials, leaving the question of regularity for future work.

In the second half of the paper we specialize to the case where $\nu $ is discrete. In this case we  give a constructive proof for the existence of a transport map, where the $c$-class function is given as a finite infimum of ``basic functions" (see \eqref{def:basic-func}) associated with the cost. The advantage of this method is that much of the geometry of the problem is revealed. In the proof, we generalize a method used by K.~Ball \cite{Ball2004elementary} for the quadratic cost, where all possible maps are parametrized by a weight vector, and the existence of the required one is shown using Brouwer's fixed point theorem. However, in contrast with the case of the classical quadratic cost function and other traditional costs, when the cost attains infinite values the set of all discrete measures with a given support, to which a measure $\mu$ can be mapped with finite cost, is given by an interesting  polytope which we call the \emph{Hall polytope} of the measure $\mu$. The condition of strong $c$-compatibility corresponds to measures with weight vectors in the interior of the polytope.
We present a thorough study of the structure and geometry of Hall polytopes (which for traditional costs are just simplices), which we use to prove Theorem \ref{thm:discrete-transport} below. An advantage of this method is that we can relax the conditions on the cost function. We do need a condition of 
$c$-regularity for the measure $\mu$ (given in Definition \ref{def:c-regular}), which for 
 the polar cost  is satisfied if, say, $\mu$ is absolutely continuous. 

\begin{thm}\label{thm:discrete-transport}
	Let $X$ be some Polish space   and $Y = \{u_i\}_{i=1}^m$. Assume $c:X\times Y\to (-\infty,\infty]$ is a measurable cost function,  $\mu\in {\mathcal P}(X)$ is $c$-regular   and $\nu= \sum_{i=1}^m \alpha_i \ind_{u_i} \in {\mathcal P}(Y)$.  Assume furthermore, that the intersection
	\[\{x\in X: c(x,u_i) <\infty \} \cap \{x\in X: c(x,u_j) <\infty \}\]
	contains an open set for each pair $u_i,u_j$. If $\mu$ and $\nu$ are strongly $c$-compatible then there exists an optimal transport plan $\pi \in \Pi(\mu, \nu)$ whose graph lies in the  
	$c$-subgradient  $\pa^c \varphi$ of a $c$-class
	function $\varphi: X\to [-\infty, \infty]$. 
\end{thm}

The case where the measures $\mu$ and $\nu$ are $c$-compatible but not strongly so, can be analyzed as well. In this case we can write $\mu = \mu_1+ \mu_2$ and $\nu = \nu_1 + \nu_2$ where $\mu_1(X) = \nu_1(Y)$ (and so $\mu_2(X) = \nu_2(Y)$), where the measures $\mu_1$ and $\mu_2$ are concentrated on disjoint sets, as are $\nu_1$ and $\nu_2$, and in such a way that any finite cost transport plan $\pi \in \Pi (\mu, \nu)$ is given as a sum of $\pi_1 \in \Pi  (\mu_1, \nu_1)$ and  $\pi_2 \in \Pi  (\mu_2, \nu_2)$. 
We illustrate this in Section \ref{sec:decomposition}.

\subsection*{Structure of the paper}
Section \ref{sec:background} is dedicated to gathering all the required definitions and notions and previous results.
In Section \ref{sec:c-compatibility} we discuss the notion of $c$-compatibility and strong $c$-compatibility together with their geometric interpretation. 
 In Section \ref{sec:arbitrary-measures} we prove Theorem \ref{thm:transport-polar-comp}.  
In Section \ref{sec:discrete-pf} we go back to the discrete case and show how one may treat it using some deep structural properties of Hall polytopes, which we establish, proving Theorem \ref{thm:discrete-transport}.
In Section \ref{sec:plansvsmaps} we specialize to the polar cost, showing that for absolutely continuous measure the optimal plan is given by a map.
In Section \ref{sec:decomposition} we discuss the case of measures which are $c$-compatible but not strongly $c$-compatible. 
For completeness an appendix \ref{appendix:c-subg} in which we review $c$-subgradients, with detailed examples and geometric intuition.
 
 \subsection*{Acknowledgment}
 The authors were supported by the 
 European Research Council (ERC) under the European Union’s Horizon 2020
 research and innovation programme (grant agreement No 770127). The second named author is grateful to the Azrieli foundation
 for the award of an Azrieli fellowship.

\section{Background and preliminary observations}\label{sec:background}

\subsection{Transport plans and maps}\label{subsec:transport-defs}

Given two measure spaces $X,\,Y$, a measurable\footnote{When referring to a function on $X \times Y$ as ``measurable'' we assume it is both measurable with respect to the product $\sigma$-algebra and its fibers $f(\cdot,y)$ and $f(x,\cdot)$ are measurable functions on $X$ and $Y$ respectively, for any $x\in X$ and $y\in Y$. } cost function $c:X\times Y \to (-\infty,\infty]$, and probability measures $\mu$ on $X$ and $\nu$ on $Y$, we say that there exists a \emph{\bf $c$-optimal transport map} between them if the following infimum is attained:
	\[\inf_{T} \int_X c(x,T(x))d\mu(x),\]
	where $T:X\to Y$ are measurable \emph{\bf transport maps}, i.e. $\nu(B)=\mu(T^{-1}(B))$ for all  measurable sets $B\subset Y$. We say that there exists a \emph{\bf $c$-optimal plan} between them if the infimum 
	\begin{equation}\label{def:total cost}\inf_\pi \int_{X\times Y} c(x,y)d\pi(x,y) 
	\end{equation}
	is attained, where $\pi\in \Pi(\mu, \nu)$, namely $\pi$ is a probability measure on $X\times Y$ satisfying
	\[\pi(A\times Y)=\mu(A),\ \ \ \pi(X\times B)=\nu(B)\]
	for all measurable sets $A\subset X$ and $B\subset Y$. Every transport map induces a transport plan supported on its graph, while not every plan is induced by a map. We denote the infimum in \eqref{def:total cost}, also called the ``total cost'', by $C(\mu,\nu)$. Due to the Kantorovich Duality Theorem \cite{Kantorovich1942transport, Kantorovich1948problem}, when $c$ is lower semi-continuous, the total cost is equal to
	\begin{equation}\label{thm:KDT} \sup_{\varphi,\psi} \left\{\int_X \varphi d\mu + \int_Y \psi d\nu :\ \varphi \in L_1(X,\mu), \, \psi \in L_1(Y,\nu){\rm ~~~admissible}\right\} \end{equation} 
	where $(\varphi, \psi)$ is called an \emph{\bf admissible pair}, if  $\varphi: X\to [-\infty,\infty]$, $\psi: Y \to[-\infty,\infty]$  satisfy
	\[\forall (x,y)\in X\times Y, \,\,\,\ \varphi(x)+\psi(y) \le c(x,y).\]
In the case where $\varphi = +\infty$ and $\psi = -\infty$ we stipulate $-\infty + \infty = -\infty$, namely in such a case the condition above holds regardless of the value of $c(x,y)$.

\subsection{The $c$-transform}\label{subsec-ctransform-cclass}
Motivated by \eqref{thm:KDT}, for every function $\psi:Y\to [-\infty,\infty]$ one may consider the largest function $\varphi$ for which $(\varphi, \psi)$ is an admissible pair, and vice versa. This gives rise to the {\bf $\mathbf{c}$-transform}, defined by
	\begin{equation}\label{eq:c-transform} 
	\psi^c(x)=\inf_y (c(x,y)-\psi(y)), \end{equation} 
	and
	\begin{equation}\label{eq:c-transform2}\varphi^c(y)=\inf_x (c(x,y)-\varphi(x)).\end{equation} 
	
	\begin{rem}
		
	Here if on the right hand side are infinities of opposite signs, which may occur only if $\psi(y) = \infty$ (as $c\neq -\infty$), we use the opposite convention, namely 
	 $-\infty + \infty = +\infty$, since when the cost $c(x,y)$ is infinite there is no restriction on the sum $\varphi(x) + \psi(y)$. In general one must be careful with sums of opposite side infinities, as there is no obvious ``rule of thumb'' that can apply everywhere. 
	 	 
	Note that for a general cost we may lose the measurability of $\varphi$ when applying the $c$-transform, as well as integrability, even under the assumption that $c$ is measurable in the strong sense we have postulated. When $c$ is continuous, however, this is less of a problem. Also, by truncating the functions and taking limits, the issue of integrability can sometimes be resolved. Nevertheless, one should be extra careful when using \eqref{thm:KDT} for a pair $\varphi, \varphi^c$ when the cost is non-traditional, and in the existing literature it is not always clear for which theorems does the non-traditional case follow from the same proof.  
	
\end{rem}
	
	When $X=Y$ and $c(\cdot,\cdot)$ is symmetric in its arguments the transforms in \eqref{eq:c-transform} and \eqref{eq:c-transform2} coincide. Hence, abusing notation, we use the same notation for both. We define the \textbf{$c$-class} as the image of the $c$-transform $\{\psi^c: \psi : Y \to [-\infty, \infty]\}$, or equivalently, as all the functions $\varphi$ such that $\varphi^{cc}=\varphi$. By definition, any function in the $c$-class is an infimum of \emph{\bf basic functions}, which are functions of the form \begin{equation}
	\label{def:basic-func} c(x)=c(x,y_0)+t
	\end{equation}
	for some $y_0\in Y$ and $t\in \RR$. It is useful to notice that the $c$-class is always closed under pointwise infimum  {(this fact is commonly known and used, see e.g. \cite{ambrosio-pratelli,villani2003topics}, and a simple proof can be found in \cite{kasia-thesis}).}

\subsection{The $c$-subgradient}	
	
Given a function $\varphi$ in the $c$-class, its $c$-subgradient is the subset of $X\times Y$ given by 
\begin{equation}\label{eq:c-subgradient}
	\pa^c\varphi =\{(x,y):\, \varphi(x)+\varphi^c(y)=c(x,y) \,\text{ and } \, c(x,y)<\infty\}.
\end{equation}

To illustrate the relevance of $c$-subgradients to the study of optimal transport, let us present a folklore argument, which can be made precise for traditional costs, and which we only use as motivation but do not claim it holds in general.  
 
In Kantorovich Duality Theorem, recalled as \eqref{thm:KDT} above, one is inclined to replace $\psi$ with the largest admissible partner of $\varphi$ (at least so long as it is measurable and in $L_1(\nu)$), and then replace $\varphi$ by $\varphi^{cc}$. In this sense, one may think of \eqref{thm:KDT} applied only to admissible pairs $(\varphi,\varphi^c)$, where $\varphi=\v^{cc}$ is in the $c$-class. However, for {\em any} $\pi \in \Pi(\mu, \nu)$ and $\varphi$ is in the $c$-class, 
\begin{align*}
 \int_X \varphi d\mu(x)+\int_Y\varphi^c(y) d\nu(y) = \int_{X\times Y}( \varphi(x)+\varphi^c(y)) d\pi(x,y) \le 
 \int_{X\times Y} c(x,y) d\pi(x,y)  
\end{align*} 
So for equality between the left and right hand side to be obtained for some (potential) $\varphi$ and (optimal plan) $\pi$, we see that $\pi$ must be concentrated on the set $\pa^c\varphi$. In other words, finding optimal plans admitting a potential is equivalent to finding {\em some} plan supported on a $c$-subgradient. While this argument is not precise (in particular, we ignored measurability and integrability assumptions, applying \eqref{thm:KDT} to a pair $(\varphi, \varphi^c)$), it constitutes the motivation behind searching for potentials in optimal transport problems.

The above observation shows the importance of the notion of the $c$-subgradient mapping. The name $c$-subgradient is connected to the fact that for the classical cost $c(x,y) = -\iprod{x}{y}$, the $c$-class consists of upper semi-continuous concave functions, the $c$-transform of $-\varphi$ is $-\L(\varphi)$, and the $c$-subgradient of $-\varphi$ at $x$ is the usual subgradient $\partial \varphi(x)$. So as not to disturb the flow of the paper, we gathered some basic facts about the $c$-subgradient, including the geometric intuition behind it, in Appendix \ref{appendix:c-subg}.

\subsection{$c$-cyclic monotonicity and $c$-path-boundedness}

The connection between optimality of a plan and some geometric information on its support is quite intuitive: if a plan is optimal, then we should not gain any profit by interchanging several portions of it. This is the idea behind the well known notion of  $c$-cyclic  monotonicity. Given a cost $c:X\times Y \to (-\infty, \infty]$, a subset $G\subset X\times Y$ is called 
 $c$-cyclically  monotone
if $c(x,y)<\infty$ for all $(x,y)\in G$, and for any $m$, any $(x_i, y_i)_{i=1}^m \subset G$, and any permutation $\sigma$ of $[m] = \{1, \ldots,m\}$ it holds that 
\begin{equation}\label{eq:ccm}  \sum_{i=1}^m c(x_i, y_i) \le \sum_{i=1}^m c(x_i, y_{\sigma(i)}).  \end{equation}

This definition seems to have been first introduced by Knott and Smith \cite{knott-smith}, as a generalization of cyclic monotonicity considered by Rockafellar \cite{rockafellar} in the case of quadratic cost. It is easy to check that if $\varphi$ is a $c$-class function then any set  $G\subset \partial^c\varphi$ is  $c$-cyclically  monotone. The theorems of Rockafellar,  Rochet and R\"uschendorf give the reverse implication, in the case of a traditional cost. Namely, when $c:X\times Y\to \RR$, a set $G\subset X\times Y$ is $c$-cyclically monotone if and only if there exists a $c$-class function such that $G\subset \partial^c \varphi$.

For non-traditional costs, this is no longer the case, and one may construct  $c$-cyclically  monotone sets which admit no potential.  In \cite{paper1}, the corresponding result for non-traditional costs is provided.  Cyclic monotonicity 
has to be replaced by a stronger notion, which we called   $c$-path-boundedness.
\begin{definition}
	Fix sets $X,\,Y$ and $c:X\times Y\to (-\infty, \infty]$. A subset $G\subset X\times Y$ will be called {\bf $c$-path-bounded} if $c(x,y)<\infty$ for any $(x,y)\in G$, and for any $(x,y)\in G$ and $(z,w)\in G$, there exists a constant $M=M((x,y), (z,w))\in \RR$, such that the following holds: For any $m\in {\mathbb N}$ and any $(x_i, y_i)_{i=2}^{m-1} \subset G$, denoting $(x_1, y_1) = (x,y)$ and $(x_m, y_m) = (z,w)$,   we have 
	\[ \sum_{i=1}^{m-1}\left( c(x_i, y_i)-c(x_{i+1}, y_{i})\right) \le M.  \]   
\end{definition}
The fact that a $c$-path-bounded set is also $c$-cyclically monotone is easy to establish (see \cite{paper1}). With this definition the main theorem of \cite{paper1} can be stated.
\begin{thm}\label{thm:path-bdd-potential}
		Let  $X,\,Y$ be sets and let $c:X\times Y\to (-\infty, \infty]$ be   given. A set $G\subset X\times Y$ is    {$c$-path-bounded} if and only if there exists a $c$-class function $\varphi$ such that $G\subset \partial^c \varphi$.
\end{thm} 

It was also demonstrated in \cite{paper1} that under certain conditions, the notions of $c$-cyclic  monotonicity and $c$-path-boundedness do coincide. One such instance, which will be used in this paper, is explained and formulated in Proposition \ref{prop:one-class-c-path-bdd} in Section \ref{sec:arbitrary-measures}. 

\subsection{Some know results about existence of optimal plans and potentials}\label{subsec:path-bdd-cases}
Having fixed a cost,  the discussion about the structure of an optimal plan naturally splits into several components. The first, which is relevant only when the cost is non-traditional, is the existence of {\em some} finite cost plan (necessary conditions will be discussed in the next section).  Further, one can ask whether an optimal plan exists.  This is the object of the next theorem, which is quoted from Villani \cite{villani-book}.

Recall that $\Pi(\mu, \nu)$ denotes the set of all probability measures on $X\times Y$ whose marginals are $\mu\in \calP (X)$ and $\nu\in \calP (Y)$, and that $c:X\times Y \to (-\infty,\infty]$ is essentially bounded with respect to $\mu$ and $\nu $
if there exist  upper semi-continuous function $a:X\to (-\infty,\infty]$,  $a\in L_1(\mu)$
and $b:X\to (-\infty,\infty]$, $b\in L_1(\nu)$ such that $c(x, y)\ge a(x) + b(y)$ for all
$x\in X,\, y\in Y$.

\begin{thm}\label{thm:villani} 
	Let $X,\, Y$ be two Polish spaces, let $\mu \in \calP (X)$ and $\nu \in \calP (Y)$. Let  $c:X\times Y \to (-\infty,\infty]$ be a lower 
	semi-continuous cost function which is essentially bounded with respect to $\mu  $ and $\nu$. Then there exists a $c$-optimal plan $\pi \in \Pi(\mu,\nu)$.
\end{thm}

Let us note that,  in the above theorem,  the existence of a plan with finite total cost is not assumed as when no finite cost plan exists, any plan (say,  $\mu \otimes \nu$) is optimal in a trivial sense.  Further,  a simple example demonstrates that without some kind of assumption on boundedness from below of the cost,  the total cost may be $-\infty$, and in this case optimal measures can be concentrated on sets which are far from being $c$-cyclically monotone. 
\begin{exm}\label{exm:minfinity}
	Let $p(x,y)=-\ln(xy-1)_+$ be the polar cost on $\R_+\times \R_+$. Let $\mu$ be a discrete probability measure on $\R_+$ given by $\mu=\sum_{n=2}^\infty \alpha_n \ind_n$, where $\alpha_n$ are such that $\sum_{n=2}^\infty \alpha_n=1$ and $\sum_{n=2}^\infty n^{3/2} \alpha_n  =\infty$. Consider transport plans of $\mu$ to itself, namely $\Pi(\mu, \mu)$.
	
	WE claim that in this case, the identity map $x\mapsto x$ is a transport plan whose total cost is $-\infty$ (in particular, it is optimal) but it is not supported on a $p$-cyclically monotone set.   Indeed, consider the measure $\pi_\mu$ on the diagonal whose projection is $\mu$. Its total cost is 
	\[\sum_{n=2}^\infty -\ln(n^2-1)\alpha_n \le -\sum_{n=1}^\infty n^{3/2} \alpha_n =-\infty.
	\]
	Clearly even for two points $(x_1, y_1) = (2,2)$ and $(x_2, y_2) = (3,2)$ it holds that 
	\[ -\ln (2\cdot 2 - 1) - \ln (3\cdot 3 -1)   =-\ln (24) > 
	-\ln (2\cdot 3 -1) - \ln (3\cdot 2-1) = -\ln (25). 
	\] 
	We thus see that an optimal plan (albeit with negative infinity cost) may have support which is not $c$-cyclically monotone.
\end{exm}

Analysing the geometric structure of an optimal plan,  after showing its existence,  is a problem which has a long history.  After Brenier \cite{Brenier}, following R\"uschendorf \cite{Ruschendorf1996c} determined the classical structure of cyclic monotonicity of optimal plans, Gangbo and McCann \cite{gangbo-mccann}	 extended the result to lower semi-continuous cost functions bounded from below. They showed that every finite optimal plan with respect to such costs lies on a $c$-cyclically monotone set.  Beiglböck, Goldstern, Maresch, and Schachermayer \cite{Schachermayer-optimal-and-better} generalised the result further by removing regularity assumptions on the cost:
\begin{thm}[See {\cite[Theorem 1.a]{Schachermayer-optimal-and-better}}]\label{thm: optimal better 1A}
	Let $X, \,Y$ be Polish spaces equipped with Borel probability measures $\mu, \nu$ and let $c : X \times Y \to [0, \infty]$ a Borel measurable cost function. Then every finite optimal transport plan is $c$-cyclically monotone.
\end{thm}

The reverse implication, that $c$-cyclic monotonicity implies optimality,  is not true in general as shown in Example 3.1 in \cite{ambrosio-pratelli}.  In \cite{Schachermayer-optimal-and-better} Theorem 1.b,  it was shown that for a measurable cost function $c$ the assumption that the ``infinity'' set $\{(x, y) : c(x, y) = \infty\}$ is a union of a closed
set and a $\mu\otimes\nu$-null set, implies that every finite $c$-cyclically monotone plan is optimal. 

Finally,  the question of the existence of a potential for the optimal plan remains.  A result in this direction was presented in \cite{Schachermayer-optimal-and-better}; it states that, with assumptions as in Theorem \ref{thm: optimal better 1A}, a finite cost plan admits a potential if and only if it is ``robustly optimal'' (see Definition 1.6. in \cite{Schachermayer-optimal-and-better}).  In particular,  their result implies that a plan which admits a potential is optimal.  In this note, our main goal is to find conditions on the pairs of measures that guarantee the existence of a potential for the optimal transport plan between them, thus guaranteeing, in fact,  robust optimality.

\section{Compatibility}\label{sec:c-compatibility}

Given two probability measures, before trying to find an element of $\Pi(\mu, \nu)$ with some good structure (say, a potential), or an optimal element with respect to the cost, one must figure out whether {\em any} element $\pi \in \Pi(\mu, \nu)$ has a finite cost.  
Clearly, if the cost function is bounded, we may find a finite cost plan between any pair of measures. However, if the cost admits the value $+\infty$, an obvious necessary condition for the existence of a finite cost plan is that every set in $(X, \mu)$ has ``enough'' points in $(Y, \nu)$ to which it can be mapped for a finite cost.

In the case of two discrete measures, this necessary condition is also sufficient, which is the subject of Hall's marriage theorem. We start with this simple case as it gives some intuition for our next steps.

\subsection{Starting point: Hall's Marriage Theorem}
  In the following motivating example, for some $(x_i)_{i=1}^m \subset X$ let $\mu=\sum_{i=1}^m \frac{1}{m}\ind_{x_i}$ be a probability  measure on $X$, and for $(y_i)_{i=1}^m \subset Y$ let $\nu=\sum_{i=1}^m \frac{1}{m}\ind_{y_i}$ be a probability measure on $Y$. Let $c:X\times Y \to (-\infty, \infty]$  be an arbitrary cost. A finite cost map is a given by a bijection $T:(x_i)_{i=1}^m \to (y_i)_{i=1}^m$, such that   $c(x_i, T(x_i))<\infty$ for all $i=1,\dots m$.
  The bijection $T$ corresponds, of course, to a permutation $\sigma:[m]\to [m]$. 
   By Birkhoff's theorem on the extremal points of bi-stochastic matrices, every transport plan $\pi \in \Pi (\mu, \nu)$ is a convex combination of permutation maps $T$.   
  
  The condition for the existence of a finite cost map/plan can be thus reformulated in a graph-theoretic way: Let $G$ be a bipartite graph with a vertex set $V=(x_i)_{i=1}^m\cup (y_i)_{i=1}^m$ and edges $E=\{(x_i,y_j):\, c(x_i,y_j)<  \infty \}$. A finite cost map $T$ corresponds a \textbf{matching} in this graph.  Hall's Marriage Theorem gives the necessary and sufficient conditions for such a matching to exist.

\begin{thm}[Hall's Marriage Theorem]\label{thm:halls-marriage}
	A bipartite graph $G$ with a vertex set $V_1 \cup V_2$, such that $|V_1|=|V_2|$, contains a complete matching  if and only if $G$ satisfies Hall's condition
	\[|N_G(S)| \ge |S| \text{ for every } S \subset V_1,\]
	where $N_G(S)\subset V_2$ is the set of all neighbors of vertices in $S$.
\end{thm} 

The condition can be reformulated in terms of the measures, as 
\[\mu(A)\le \nu(\{y: \exists x\in A,\,\, c(x,y)<\infty\}) \]
for any $A\subset X$, 
or, equivalently, 
\[\mu(A)+ \nu(\{y: \forall x\in A,\,\, c(x,y)=\infty\}) \le 1.\]

In fact, in this discrete and finite case, once we have determined the existence of a finite cost map, we may consider, among the finite number of possible matchings, the one with minimal cost (there may, of course, be more than one). 
It is then not hard to show (and will follow from our results as well) that this resulting optimal plan must lie on a $c$-subgradient of a $c$-class function. (This fact follows from a variation of a theorem of R\"uschendorf \cite{Ruschendorf1996c},  see also \cite{paper1}.)

\subsection{The $c$-compatibility condition}

The continuous counterpart for Hall's condition is an obvious necessary condition for the existence of a finite cost plan. 

\begin{definition}\label{def:c-comp-finite}
	Let $X,\,Y$ be measure spaces and  $c:X\times Y \to (-\infty, \infty]$ be a measurable cost function. We say that two probability measures $\mu\in {\mathcal P}(X)$ and $\nu \in {\mathcal P}(Y)$ are  \emph{\textbf{$c$-compatible}}  if for any measurable $A\subset X$ it holds that
	\[ \mu(A) + \nu(\{y: \forall x \in A,\,\,c(x,y) = \infty\}) \le 1.\] 
\end{definition}

It is not hard to check that $c$-compatibility is in fact a symmetric notion, and the above condition holds if and only if for any $B\subset Y$ we have 
\[ \nu(B) +  \mu(\{x: \forall y\in B,\,\,c(x,y) = \infty\})\le 1.\]

Indeed, to get the latter we let $A = \{ x: \forall y\in B,\,\,c(x,y) = \infty\}$, in which case $B\subset \{ y: \forall x \in A,\,\,c(x,y) = \infty\}$. Applying the assumed inequality, we get  
\[  \nu(B) +  \mu(A) \le \nu(\{y: \forall x \in A,\,\,c(x,y) = \infty\}) + \mu(A) \le 1.\] 
 
The fact that any plan $\pi \in \Pi(\mu, \nu)$ which has finite cost must be concentrated on the finiteness set \[S = \{ (x,y): c(x,y)<\infty\} \subset X\times Y\] implies the necessity of the condition, as is given in the following lemma.

\begin{lem}\label{lem:pre-strassen}
	Let $X,\,Y$ be measure spaces and  $c:X\times Y \to (-\infty, \infty]$ be a measurable cost function. Given $\mu\in {\mathcal P}(X)$ and $\nu \in {\mathcal P}(Y)$, assume there exists $\pi \in \Pi(\mu, \nu)$   which is concentrated on  	$S = \{(x,y)\in X\times Y:\ c(x,y)<\infty\}$. Then $\mu$ and $\nu$ are $c$-compatible. 
\end{lem}

\begin{proof}
	Let $A\subset X$. 	As  $\pi \in \Pi(\mu, \nu)$, we know
	$\mu(A) = \pi(A\times Y)$, and by assumption, $\pi(A\times Y) = \pi ((A \times Y)\cap S)$. Similarly, 
	\[ \nu(\{y: \forall x \in A,\,\,c(x,y) = \infty\})   = \pi ((X\times\{ y: \forall x \in A\,\,,\,\,c(x,y) = \infty\})\cap S ).\] 
	However, these two sets are disjoint, since if $(x,y)\in S$ then $c(x,y) < \infty$, so if $x\in A$ then clearly $y$ does not satisfy that 
	for all $x\in A$, $c(x,y) = \infty$. Therefore, the $\pi$-measures of the two sets sum to at most $1$.  
\end{proof}

It is useful to know that in certain situations the $c$-compatibility condition is also sufficient for the existence of a finite cost plan; such is the case when the finiteness set $S$ is closed. One may then use the following theorem of Strassen  \cite{Strassen1965existence}.

\begin{thm}[Strassen]\label{thm:strassen}
	Let $X,\,Y$ be complete separable metric measure spaces and let  $S$ be a non-empty closed subset of $X \times Y$. Given $\mu \in {\mathcal P}(X)$ and $\nu \in {\mathcal P}(Y)$, there exists  $\pi \in \Pi(\mu, \nu)$ which is supported on  	 $S$   if and only if for all open $B\subset Y$
	\begin{equation}\label{eq:strassen}
	\nu (B) \le \mu (P_X(S\cap (X\times B))),
	\end{equation}
	where $P_X$ is a projection onto $X$.
\end{thm}

In the case of a non-traditional cost $c$, the relevant set $S$  considered in Lemma \ref{lem:pre-strassen} is not necessarily closed. 
If $S$ is closed, and $c$ is bounded on it, %
then the condition in Strassen's Theorem is sufficient for the existence of a finite-cost transport plan. In some cases, one may use this together with the theorems stated in Section \ref{subsec:path-bdd-cases} and the results from \cite{paper1} to show that a minimizing plan exists and is concentrated on the graph of a $c$-subgradient. An example of such reasoning for some explicit cost functions will appear in the forthcoming \cite{paper3}.

However, for certain important costs, and in particular for the polar cost $p$ defined in $\eqref{def:polar-cost}$ which serves as a motivating example for this study,  the set $S$ of {\em finite-cost pairs} is not closed. 

To illustrate the problem, let us give an example of two measures on intervals which are $c$-compatible (we will use the one dimensional polar cost) but do not admit any plan supported on the finiteness set $S$. %

\begin{exm}\label{example:ccompatible-not-enough} 
	Consider once more the polar cost $p(x,y) = -\ln (xy-1)_+$ on  $\RR^+\times \RR^+$. Its finiteness set is $S = \{ (x,y): xy >1\}$. Let $\gamma$ be the uniform measure on the set $S_1 = \{(x,1/x)\in \RR^2: x\in [1/2,2]\}$ and let $\mu$ be its marginal on the first coordinate and $\nu$ its marginal on the second coordinate. 
	
	It is not hard to check that the measures $\mu$ and $\nu$ (which are the same measure) are $p$-compatible. Indeed, let $A\subset \RR^+$ be open, note that 
	\[P_X((\RR^+\times A)\cap S) = \cup_{y\in A}(1/y,\infty) = (1/\sup(A), \infty).\]
	Additionally, for any number $\alpha \in [1/2,2]$ we have, by definition, that $\nu([1/2, \alpha])=\mu([1/\alpha, 2])$. Combining these observations with the continuity of $\mu$ and $\nu$ we see that the measures are polar compatible
	\[\nu(A)\le \nu([1/2,\sup (A)]) = \mu([1/ \sup (A),2] )=\mu([1/ \sup (A), \infty))=\mu(P_X((A\times X)\cap S)). \]

	We turn to show that there is no transport plan $\pi \in \Pi(\mu, \nu)$ supported on $S$. 
	Assume towards a  contradiction that there exists such a %
	transport plan $\pi$. 
	In particular, this implies that there exists some rectangle $B=[x_1,x_2]\times [y_1,y_2] \subset S$ of positive measure. By the definition of $S$, we have that $x_1 y_1 >1$. %
	As $\pi$ is supported in $S$  we see that
	\[\mu([1/2,x_1]) = \pi([1/2,x_1]\times [x_1^{-1},2])\le \nu([x_1^{-1},2]) = \mu([1/2,x_1])\]
	where the last equality follows from the definition of $\mu$ and $\nu$. We thus have equalities all along. Similarly,  
	\[ \nu([1/2,x_1^{-1}]) = \pi([x_1,2]\times [1/2,x_1^{-1}])\le \mu([x_1,2]) = \nu([1/2,x_1^{-1}]). \]
	So we conclude that $\pi([1/2,x_1]\times [x_1^{-1},2]) +\pi([x_1,2]\times [1/2,x_1^{-1}]) = \mu([1/2,x_1])+ \mu([x_1,2]) = 1,$ 
	that is, $\pi$
	is supported on $[1/2,x_1]\times [x_1^{-1},2] \cup [x_1,2]\times [1/2,x_1^{-1}])$, which  is a contradiction to the fact that $\pi(B)>0$.
\end{exm}
\begin{figure}[H] 
	\centering
	\includegraphics[width=0.3\textwidth]{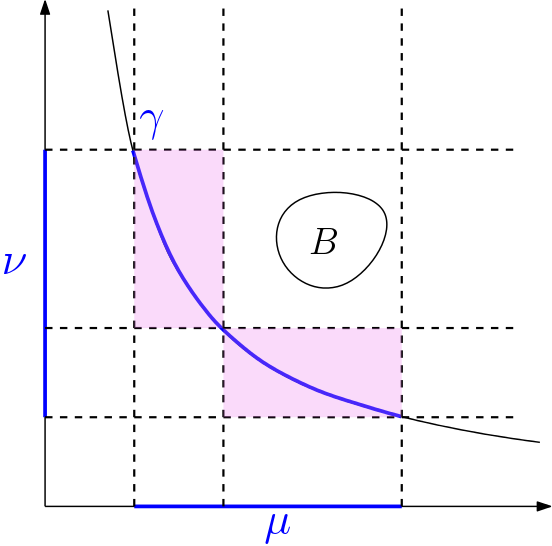}
		\caption{A schematic drawing of Example \ref{example:ccompatible-not-enough}.\label{fig:example35}}
\end{figure}

\subsection{The Hall polytope}\label{subsec:Hallpol}

Let us consider a special case, which will be the focus of Section \ref{sec:discrete-pf}, when one of the measures is discrete and the other one arbitrary. In such a case, the compatibility condition can be realized geometrically by a polytope, which we call the Hall polytope. We use $\Delta_m = \{\alpha \in \RR^m: \alpha_i \ge 0,\, \sum_{i=1}^m \alpha_i = 1\}$ to denote the $(m-1)$-dimensional simplex.

\begin{definition}\label{def:hall-poly}
	Let $X$ be some measure space, and $Y = \{u_i\}_{i=1}^m$. Assume $c:X\times Y\to (-\infty,\infty]$ is a measurable cost function, 
	and let $\mu$ be a probability measure supported on $\{x\in X: \exists i\in [m],\,\,  c(x,u_i)<\infty\}$. Define the  \emph{\textbf{Hall polytope}} associated with $(u_i)_{i=1}^m$  and $\mu$ by 
	\[ P = P((u_i)_{i=1}^m, \mu) = \bigcap_{I\subset [m]} \{ \alpha\in \Delta_m: \sum_{i\in I} \alpha_i \le\mu (A_I)\},\]
	where
	\[ A_I :=  \{x\in X:\ \exists i\in I\ c(x,u_i)<\infty\}.\] 
\end{definition}

Note that the definition implies that $\mu$ and $\nu= \sum_{i=1}^m \alpha_i \ind_{u_i}$ are $c$-compatible
if and only if $\alpha \in P((u_i)_{i=1}^m, \mu)$. 

We get back to this definition, and present a careful study of the resulting polytopes, in Section \ref{sec:discrete-pf}.

\subsection{Strong $c$-compatibility}

We saw in Example \ref{example:ccompatible-not-enough} that $c$-compatibility is not a sufficient condition for the existence of a finite cost plan. In fact, we will see in Example \ref{ex:no potential} that there exist $c$-compatible measures which do admit a finite cost plan but not a potential. 
Therefore, we consider a slight strengthening of $c$-compatibility, which will ensure that the existence of a finite cost plan implies the existence of a potential. 
 We call this condition strong $c$-compatibility, and it amounts to asking for a strict inequality in the defining inequalities.

\begin{definition}\label{def:strong-c-comp1}
		Let $X,\,Y$ be measure spaces and  $c:X\times Y \to (-\infty, \infty]$ be a measurable cost function. We say that two probability measures $\mu\in {\mathcal P}(X)$ and $\nu \in {\mathcal P}(Y)$ are   \emph{\textbf{strongly $c$-compatible}}  if they are $c$-compatible and  for any measurable $A\subset X$ with $0<\mu(A)<1$ it holds that
	\[ \mu(A) + \nu(\{y: \forall x \in A,\,\,c(x,y) = \infty\}) < 1.\] 
\end{definition}

The motivation for this specific strengthening of the condition of $c$-compatibility  is twofold: First, if two measures are  $c$-compatible and not strongly $c$-compatible, this means that there exists a decomposition of the transport problem into two sub-problems (see Section \ref{sec:decomposition}). Indeed, this is quite clear from the definition: if some set $A$ of measure $\mu(A)\in (0,1)$ satisfies the equality 
\[ \mu(A) + \nu(\{y: \forall x \in A,\,\,c(x,y) = \infty\}) = 1,\]
then letting $B = \{y: \forall x \in A,\,\,c(x,y) = \infty\}$ we see that $A$ must be mapped to $Y\setminus B$ (and they have the same measure) and the preimage of $B$ must be $X\setminus A$. That is, the original transport problem is in fact decomposed into two disjoint transport problems.

 Second, in the discrete setting of Section \ref{subsec:Hallpol},   strong $c$-compatibility corresponds to the weight vector $\alpha$ residing in the interior of the Hall polytope, which makes for an elegant assumption.  
 
 We stress that strong $c$-compatibility is
  not a necessary condition, only $c$-compatibility is. Even if one of the measures is discrete, it could be that the Hall polytope has an empty interior, but good transport maps, admitting a potential, exist.%

\subsection{The geometric meaning of strong $c$-compatibility}

It will be very useful to rephrase the condition of strong $c$-compatibility in terms that are more geometric. In fact, looking back at the proof of the symmetry of the notion of $c$-compatibility, it seems evident that we do not need to assume an inequality $\mu(A) + \nu(\{y: \forall x \in A,\,\,c(x,y) = \infty\}) \le 1$ (or a strict inequality, in the strong $c$-compatibility assumption) for all sets $A$, and it suffices to consider sets of the form $\{ x: \forall y \in B,\,\,c(x,y) = \infty\}$. To make this observation more precise, we introduce the notion of the $c$-dual of a set.

\begin{definition}[$c$-duality]
	Let $X,\,Y$ be two sets and let $c:X\times Y \to (-\infty,\infty]$. Fix $t\in (-\infty, \infty]$ (which will be omitted in the notation as it is a fixed parameter).  
	For $K\subset X$ define the \emph{$c$-dual} set of $K$ as
	\[ K^c =\bigcap _{x\in K} \{ y\in Y:\, c(x,y)\ge t \}= \{ y\in Y:\, \inf_{x\in K} c(x,y) \ge t \}.\]
\end{definition}

It will be convenient to assume $X = Y$ and that the cost is symmetric, and as this is the case relevant for this note, we restrict to this case. However, the reader will find it easy to generalize to the case where $X\neq Y$, in which case there are two different ``$c$-duality'' operations, one mapping sets in $X$ to sets in $Y$, and one mapping sets in $Y$ to sets in $X$, similarly to the $c$-transform. 

Let us point out that for the polar cost $p(x,y)=-\ln (\sp{x,y}-1)_+$ and $t = \infty$,  the set $K^p$ is the well known polar set $K^\circ$. Indeed, we have that $\inf_{x\in K} p(x,y)=\infty$ if and only if $ \sup_{x\in K} \sp{x,y}\le 1$. For the classical cost $c(x,y) = -\iprod{x}{y}$ and $t = -1$, we also get the polarity map.  

\begin{rem}
	If one adds the assumptions that $X$ and $Y$ are measure spaces and that the cost is upper semi-continuous, it follows that for a fixed $x$, say, the set $\{y:\, c(x,y)\ge t\}$ is closed, and hence so is $K^c$.
\end{rem}

Having defined an operation on sets, let us notice some basic properties. 

\begin{lem}\label{lem:c-dual properties} For every $K,L\subset X$, the following hold
	\begin{enumerate}[(i)]
		\item $K\subset (K^c)^c=K^{cc}$,
		\item if $L\subset K$ then $K^c \subset L^c$,
		\item $K^c=K^{ccc}$.
	\end{enumerate}
\end{lem}

\begin{proof}
	\textit{(i)} This follows directly from the definition. If $x\in K$ and $y\in K^c$ then $c(x,y)\ge t$ so that $x\in K^{cc}$.
	
	\textit{(ii)} Assume that $L\subset K$, and $y\in K^c$, then $c(x,y)\ge t$ for all $x\in K$ and in particular for all $x\in L$, so $y\in L^c$.
	
	\textit{(iii)} From \textit{(i)} we know that $K \subset K^{cc}$, so from 
	\textit{(ii)} we get  $K^c \supset K^{ccc}$. On the other hand, applying \textit{(i)} directly to $K^c$ we get $K^{ccc}\subset K^c$, and equality is obtained. 
\end{proof}

The similarity of $c$-duality to the $c$-transform is apparent. We are thus motivated to define the $c$-class of sets, on which the $c$-duality is an order reversing bijection. In order to avoid confusion, as we suppressed $t$ in the notation, we restrict the next definition to $t=\infty$, the case relevant for this note.

\begin{definition}[$c$-class and  $c$-envelope]
Fix $t = \infty$.	The $c$-class of sets consists of all closed sets $K\subset X$ such that there exists some $L\subset X$ with $K=L^c$.
	For any set $K\subset X$ we define its $c$-envelope as the set $K^{cc}$,  which is the smallest $c$-class set containing $K$.
\end{definition} 

Let us note again that for the polar cost and $t = \infty$, the $p$-class consists of closed convex sets containing the origin, and the $p$-envelope is the polar convexification operation $K \mapsto K^{\circ \circ} =\overline{\conv \{0,K\}}$.

Our first observation is that in Definitions \ref{def:c-comp-finite} and \ref{def:strong-c-comp1} it is sufficient  to consider $c$-class sets, for $t= \infty$, instead of all measurable sets.

\begin{lem}\label{lem:polar-compatible1}
	Let $c:X\times Y \to (-\infty, \infty]$ be an upper semi-continuous symmetric cost function. Two probability measures $\mu\in {\mathcal P}(X)$ and $\nu\in {\mathcal P}(Y)$ are  $c$-compatible if and only if for every set $K=K^{cc}\subset X$ in the $c$-class we have \[\nu(K^c)\le 1 - \mu(K).\] They are strongly $c$-compatible if and only if in addition when $\nu(K^c)\neq 0,1$ we have \[\nu(K^c)< 1 - \mu(K).\] 
\end{lem}

\begin{proof}
	If $\mu$ and $\nu$ are  $c$-compatible then in particular $\nu(K^c) \le \mu(\{x: \, \inf_{y\in K^c}  c(x,y)<\infty \} )$, which can be rewritten as $\nu(K^c)\le \mu(X \setminus K^{cc})=\mu(X \setminus K)$. 
	
	For the other direction let $A\subset X$ be a measurable set, and consider the set $K = A^{cc}$. Then \[ \{y:\inf_{x\in A} c(x,y)=\infty \} = \{y: \forall {x \in A}\,\, c(x,y)=\infty \}= A^c =K^c.\]
	The last equality holds due to Lemma \ref{lem:c-dual properties} \textit{(iii)}. 
	Thus, using the condition on $c$-class sets and Lemma \ref{lem:c-dual properties} \textit{(i)}, we get
	\[ \nu(A^c)=\nu(K^c) \le 1- \mu(K) = 1-\mu(A^{cc}) \le 1-\mu(A),\]
	so that $\mu$ and $\nu$ are $c$-compatible.

	Similarly, two probability measures $\mu$ and $\nu$ are strongly  $c$-compatible if and only if they are $c$-compatible and for all $c$-class sets $K\subset X$ such that $\nu(K^c)\neq 0,1$ we have \[\nu(K^c)< 1 - \mu(K).\] 
	This follows from the same proof, the only difference being if $A\neq A^{cc}=K$, one gets a strong inequality by $\nu(A^c)\le 1-\mu(A^{cc}) < 1- \mu(A)$, which follows by Lemma \ref{lem:c-dual properties} \textit{(i)}. 	\end{proof}

In the next lemma we show that the strong $c$-compatibility of two measures implies a vital condition on the distribution of the transport plan between them. 

\begin{lem}\label{c-non-empty}
	Let $\mu$ be a probability measure on $X$, $\nu$ a probability measure on $Y$, and $\pi \in \Pi(\mu, \nu)$ a finite cost plan, with respect to the symmetric cost $c:X\times Y \to (-\infty, \infty]$. Then $\mu$ and $\nu$ are strongly $c$-compatible if and only if for every $c$-class set $K$ such that $\nu(K^c)\neq 0,1$, we have that
	\[ \pi( (X\setminus K)\times (Y\setminus K^c))>0.
	\]
\end{lem}

\begin{proof}
	First we note that the existence of a finite cost plan $\pi \in \Pi(\mu, \nu)$	
	implies $c$-compatibility (see Lemma \ref{lem:pre-strassen}). Thus, under our assumptions, strong $c$-compatibility is equivalent, by Lemma \ref{lem:polar-compatible1}, to the fact that for every $c$-class $K$ with $\nu(K^c)\neq 0,1$  we have that $\nu(K^c )< 1- \mu(K) = \mu (X\setminus K)$. 
	Since $\pi\in \Pi(\mu, \nu)$ this can be rewritten as, for  $\nu(K^c)\neq 0,1$, 
	\[ \pi(X \times K^c)<\pi ((X \setminus K)\times Y), \]
	and if $\nu(K^c) = 1$ then $\mu(K) = 0$. 
	Note that as $\pi$ has finite cost, it is concentrated on $ S=\{(x,y): c(x,y)<\infty\}$, and so for $(x,y)$ in the support of $\pi$, if $y\in K^c$ then we must have $x\not \in K$. In particular, from the point of view of the measure $\pi$, the set on the left hand side is contained in the set on the right hand side. We can thus rewrite the first inequality as  
	\[0< \pi (((X \setminus K)\times Y) \setminus (X \times K^c))=\pi ((X\setminus K) \times (Y\setminus K^c)).
	\]
	completing the proof of the statement claimed.
\end{proof}

\section{Transportation of measure}\label{sec:arbitrary-measures}

Let us recall our main theorem, to be proved in this section.  
\begin{manualtheorem}{1.1}
	Let $X = Y$ be a Polish space, and $c:X\times Y \to (-\infty,\infty]$ be a continuous and symmetric cost function, essentially bounded from below with respect to   $\mu\in {\mathcal P}(X)$ and $\nu\in {\mathcal P}(Y)$. Assume $(\mu, \nu)$ are strongly $c$-compatible, and $C(\mu, \nu) <\infty$. 
	Then there exists a $c$-class function $\varphi$ and an optimal transport plan $\pi \in \Pi(\mu, \nu)$ concentrated on $\partial^c \varphi$. 
\end{manualtheorem}

In order to prove the theorem we will use a combination of Theorems \ref{thm:path-bdd-potential}, \ref{thm:villani}, and  \ref{thm: optimal better 1A}. We will show that once we have an optimal transport plan supported on a $c$-cyclically monotone set then it must be $c$-path-bounded. This will follow from an observation presented in \cite{paper1} which states that indeed in some special cases $c$-cyclic monotonicity implies $c$-path-boundedness. In order to formulate the condition let us introduce some notation.

We consider a directed graph, associated with a cost function $c:X\times Y \to (-\infty,\infty]$ and a set $G \subset S = \{(x,y): c(x,y)<\infty\} \subset X\times Y$, 
in which the vertices are elements of $G$ and there is a directed edge from $(x,y)$ to $(z,w)$ if $c(z,y)<\infty$. Since $G\subset S$ we may say that for every point in $G$ there is an edge (loop) with this point as a start and end vertex. 

The directed graph induces a (transitive) relation on points in $G$, namely $(x,y)\prec (z,w)$ if there is a directed path from $(x,y)$ to $(z,w)$. We then define an equivalence relation $\sim$ on elements of $G$ where we say that $(x,y) \sim (z,w)$ if $(x,y)\prec (z,w)$ and $(z,w)\prec (x,y)$, i.e. there is a directed cycle passing through both points.  To the best of our knowledge,  this equivalence relation was first mentioned in \cite[Chapter 5, p.75]{villani-book} and studied in \cite{bianchini-caravenna, Schachermayer-optimal-and-better,paper1}.  The following proposition was proved (with a different formulation) in \cite{Schachermayer-optimal-and-better} and then in \cite{paper1}.

\begin{prop}\label{prop:one-class-c-path-bdd}
Let $c: X \times Y \to (-\infty ,\infty]$ be some cost function and let $G\subset X \times Y$ be a $c$-cyclically monotone set. Assume that all points in $G$ belong to one equivalence class of the equivalence relation $\sim$ defined above. Then $G$ is $c$-path-bounded. 
\end{prop}

With this proposition in hand, our goal is to show that if $\pi$ is a finite cost plan between two strongly $c$-compatible measures, then we can find a set $G$, on which $\pi$ is concentrated, such that all of  points in $G$ are in one equivalence class of $\sim$.

\begin{prop}\label{prop:c-one-equivalence-class}
	Let $X,\,Y$ be two Polish spaces,   $\mu \in \calP(X), \nu\in \calP(Y)$ and assume $(\mu, \nu)$ are strongly $c$-compatible. 
	Let $\pi\in \Pi(\mu,\nu)$ be a finite cost transport plan from $\mu$ to $\nu$. 
	Then there exists a set $G$ on which $\pi$ is concentrated such that 
	all the points in $G $ are in one equivalence class of $\sim$. 
\end{prop}

\begin{proof} 
	Let $G_1$ denote the support of $\pi$, and let $G_0$ denote the set $G_1 \cap \{(x,y)\in X\times Y: c(x,y)<\infty\}$. 
	Fix a point $(x,y)\in G_0$. We shall show that the set of points
	$(z,w)\prec (x,y)$ in $G_0$ is of $\pi$-measure one, as is the set of points $(z,w)$ such that $(x,y)\prec (z,w)$. The intersection of these two sets will also be of measure one, and we denote it by $G$. We will then explain why this $G$ fulfills the requirements of the proposition.
	
	Consider $H\subset G_0$ consisting of all points $(a,b) \prec (x,y)$. Assume towards a contradiction that $\pi (H) <1$. 
	Note that $\pi(H) >0$ since $(x,y)\in S_0$ and in the support of $\pi$, so that for any neighborhood $U$ of $(x,y)$ we have $\pi(U) >0$. Picking a small enough neighborhood $U$, we know that if $(z,w) \in U$ then $c(z,y)<\infty$ and so $(z,w) \in H$ (it may be that $\pi(x,y)>0$ and that $U$ consists of this one point alone). 
	
	Since $\pi(H) <1$ there is some $(z,w)\not\in H$ which is a density point of $\pi$, that is, for any neighborhood $V$ of $(z,w)$ one has $\pi(V) >0$.  
	
	If $z \in  (P_XH)^{cc}$ then $ (P_XH)^c=(P_XH)^{ccc}\subset \{z\}^c$ by Lemma \ref{lem:c-dual properties} \textit{(ii)} and \textit{(iii)}. Therefore $Y \setminus \{z\}^c \subset Y \setminus (P_XH)^c= \{v:\, \exists u\in P_X H, \, \, c(u,v)<\infty\}$. Further, since $(z,w)\in G$ we have that $c(z,w)<\infty$ and hence $w\in Y \setminus \{z\}^c$. But this means that $w \in Y\setminus (P_X H)^c$ and there exists some $a \in P_X H$ (and $b$ such that $(a,b)\in G_0$) with $c(a,w)<\infty$.
	Therefore $ (z,w)\prec (a,b)$, and by transitivity $(z,w)\prec (x,y)$, a contradiction.

	We may therefore assume that $(z,w)$ is such that $z \notin (P_XH)^{cc}=:K$. Since $K$ is a closed set in  $X$, there is a     neighborhood of $z$ which does not intersect $K$, and therefore we can find a neighborhood $V$ of $(z,w)$ which is of positive $\pi$ measure (as $(z,w)$ is a density point) and such that its projection onto $X$ does not intersect $K$. 
	Note that this implies in particular that $0<\mu(K)=\pi(K\times Y)<1$, since $H\subset K \times Y$ and $ V\cap (K \times Y) = \emptyset$. 
	We may therefore use Lemma \ref{c-non-empty} to deduce that 
	\[\pi((X\setminus K)\times (Y \setminus K^c)) > 0.\] 
	In particular there exists some point $(e,f)\in G_0$ Such that $e\not\in K$ and $f\not\in K^c$. 
	The fact that $e\not\in K$ means in particular that $(e,f) \not \in H$. The fact that $f\not\in K^c=(P_X H)^{ccc}=(P_X H)^c$ implies that $f\in \{v:\, \exists u \in P_X H \, c(u,v)<\infty \}$. Hence there is some point $a \in P_X H$ (and $b$ such that $(a,b)\in H$) such that $c(a,f)<\infty$, which means  that $(e,f) \prec (a,b ) \prec (x,y)$, thus contradicting the fact that $(e,f) \not\in H$. We conclude that the set $H$ satisfies $\pi(H) = 1$.
	
	Similarly we consider $F\subset G_0$ consisting of all points $(a,b)$ such that $(x,y)\prec (a,b)$. Using the same argument as above we get that $\pi(F)=1$.
	
	Hence, we found sets $F$ and $H$ of $\pi$-measure one. Let $G = F\cap H$, every point $(z,w)\in G$ satisfies that there is a directed path, going through points in $G_0$, between it and $(x,y)$. We now claim that these directed paths only go through points in $G$ itself. Indeed, consider a cycle (in $G_0$) which includes $(x,y)$ and $(z,w) \in G$. 
	The existence of this cycle implies that every point on it belongs to both $H$ and $F$, by the definition of the relation $\sim$, so that the whole cycle consists of points in $G$. The proof is now complete. 
	
\end{proof}

\begin{proof}[Proof of Theorem \ref{thm:transport-polar-comp}]
By assumption, $C(\mu, \nu)<\infty$, and we may use Theorem \ref{thm:villani}, the assumptions of which are satisfied, to find a $c$-optimal 
plan $\pi \in \Pi(\mu,\nu)$. By Theorem \ref{thm: optimal better 1A},  the plan $\pi$ is concentrated on some $c$-cyclically monotone set $G_1$.  
Proposition \ref{prop:c-one-equivalence-class} implies that $\pi$ is also concentrated on some set $G_2$ such that all points in $G_2$ are in one equivalence class of the relation $\sim$ defined above. Let $G = G_1\cap G_2$, then $G$ is a $c$-cyclically monotone set such that all of its elements lie in one equivalence class, therefore, by Proposition \ref{prop:one-class-c-path-bdd} the set $G$ is $c$-path-bounded. Finally, we use 
 Theorem \ref{thm:path-bdd-potential} which implies that a $c$-path-bounded set admits a potential, to find some $c$-class function $\varphi$ such that $G\subset \partial^c \varphi$.
 
 We have thus determined that there exists a $c$-optimal plan $\pi$ which is concentrated on $\partial^c \varphi$ for some $c$-class $\varphi$, as needed. 
\end{proof}

	\section{Transportation to a discrete measure} \label{sec:discrete-pf}

In this section we present a different approach to the problem of finding transport maps which lie on $c$-subgradients of functions. We consider the case where one measure is arbitrary (we will add some mild assumptions on it, connected with the cost, later on) and the second measure is discrete. As explained in Section \ref{subsec:Hallpol}, fixing the support of $\nu$ to be the set $\{y_i\}_{i=1}^m$, a necessary condition for the existence of a finite cost plan $\pi \in \Pi(\mu, \nu)$ is that the weight vector $\alpha\in \Delta_m$ associated with the probability measure $\nu=\sum_{i=1}^m \alpha_i \ind_{u_i}$ lies in the Hall polytope
\[ P = P((u_i)_{i=1}^m, \mu) = \bigcap_{I\subset [m]} \{ \alpha\in \Delta_m: \sum_{i\in I} \alpha_i \le\mu (A_I)\},\]
where
\[ A_I :=   \{x\in X:\ \min_i c(x,u_i)<\infty\}.\] 
So, our main objective is to show that indeed, for a measure $\nu$ corresponding to a weight vector in the polytope, a finite cost transport plan exists, and further, it is supported on the $c$-subgradient of some $c$-class function. We are able to do this under very general assumptions on the measure $\mu$, and provided $\alpha$ lies in the interior of the polytope (this is Theorem \ref{thm:discrete-transport}). Let us introduce the notion of $c$-regularity of a measure, which will be important for the construction given in this section. %
Roughly speaking, a measure is $c$-regular if it gives $0$-measure to sets where two different	 basic functions $c(\cdot ,y_1)+a_1$ and $c(\cdot ,y_0)+a_0$, coincide and equal some finite number.

\begin{definition}\label{def:c-regular}
	Let $X,\,Y$ be  measure spaces and let  $c:X\times Y \to (-\infty,\infty]
	$ be a measurable cost function, and $\mu$ a probability measure on $X$. If for any $y_1\neq y_0 \in Y$  and $t\in \RR$ 
	\[\mu(\{c(z,y_1) -  c(z,y_0) = t\})=0,\]
	then we say that $\mu$ is a \emph{\textbf{$c$-regular measure}}. 
\end{definition} 

For example, when the cost is such that  $\{ z: c(z,y_1) -  c(z,y_0) = t\}$ is of lower dimension, and the measure is absolutely continuous, the $c$-regularity property is satisfied. 

\subsection{Building transport maps}\label{subsection: building transport maps}

The idea of the proof is to manually 	
  construct functions whose $c$-subgradient is a transport map of a $c$-regular measure $\mu$ to a certain discrete measure $\nu$. We will consider basic functions and use the fact that the $c$-class is closed under the pointwise infimum. Formally, we have the following lemma. 
\begin{lem}\label{fact:c-subgradient-of-discrete-inf}
	Let $X,\,Y$ be measure spaces and let   $c:X\times Y \to (-\infty, \infty]$ be a measurable cost function. Fix a set of vectors $(u_i)_{i=1}^m\subset Y$  and let $\mu$ be a $c$-regular probability measure on $X$, which is supported on the set $\{x\in X: \exists i\in [m] {\rm~with~} c(x,u_i) <\infty\}$. 
	Given   numbers $(t_i)_{i=1}^m\subset \R$ let
	\[\varphi(x) = \varphi_{(u_i), (t_i)}(x)  = \min_{1\leq i \leq m} \left(c(x,u_i) + t_i\right)\]
	be a function in the $c$-class and denote $U_i = \left\{ x\in X:\ \arg\min_{1\leq j \leq m} \left(c(x,u_j) + t_j \right)= i \right\}
	$. 
	Then the mapping $T$, defined to be equal to $u_i$ on the set $U_i$, is well defined $\mu$-almost everywhere and satisfies that $T(x)\in \pa^c \varphi(x)$ for all $x$ in the support of $\mu$. Moreover, it transports $\mu$ to the measure $\nu =\sum_{i=1}^m \alpha_i \ind_{u_i}$ on $Y$, where
	$
	\alpha_i = \mu\left(U_i  \right).
	$
\end{lem}

\begin{proof}
	Let $\v(x)$ be the function defined in the statement and note that
	it induces a partition of $X$ into $m$ sets $\{U_i\}_{i=1}^m$, where
	\[U_i = \{x\in X:\ \arg\min_{1\leq j \leq m} \left(c(x,u_j) + t_j\right) =i\}.\]
	By the definition of the $c$-subgradient given in \eqref{eq:c-subgradient}, $u_i \in \pa^c \varphi(x)$ for all $x\in U_i$. Let $T: X \to Y$ be the map given by $T(x)=u_i$ for all $x \in U_i$, so indeed $T(x)\in \pa^c \varphi(x)$. For $\mu$ which is $c$-regular, the intersections of the sets $U_i$ are of zero measure and thus $T$ is well defined $\mu$ almost everywhere.
	
	Clearly, the map $T$ transports the measure $\mu$ on $X$ to the measure $\sum_{i=1}^m \mu(U_i)\ind_{u_i}$.
\end{proof}

\begin{rems}
	\textit{(i)} In general, the partition to sets $\{U_i\}$ as above is not disjoint, so without the additional assumption of $c$-regularity of $\mu$ the map $T$ is not well-defined. 
	
	\textit{(ii)} Since we may add a constant to all $(t_i)_{i=1}^m$ without changing the $c$-subgradient, we will assume that $t_i\ge 0$. Thus, given a finite set $(u_i)_{i=1}^m\subset Y$, it will be convenient for us to consider the family of functions
	\[\varphi(x) =  \min_{1\leq i \leq m} \left(c(x,u_i) - \ln(t_i)\right),\]
	with $t =(t_i)_{i=1}^m$ in the $m$-dimensional simplex $\Delta_m$. 
\end{rems}

Lemma \ref{fact:c-subgradient-of-discrete-inf} guarantees that given a $c$-regular measure $\mu$ and points $(u_i)_{i=1}^m \subset Y$, the map $\varphi_{(u_i), (t_i)}$ induces a transport map $T:X \to Y$  mapping $\mu$ to $\sum \alpha_i \ind_{u_i}$. This simple idea will be very important in proving Theorem \ref{thm:discrete-transport}, and the bulk of the proof lies in
analyzing which weights $\alpha_i$ can be attained. In the classical case of the quadratic cost it was proved by K. Ball that all weight vectors $\alpha = (\alpha_i)_{i=1}^m \in \Delta_m$  can be attained \cite{Ball2004elementary}, from which he then obtained the Brenier theorem for all absolutely continuous measures $\mu$ and compactly supported $\nu$ using a limiting argument. 

In contrast, in the case of non-traditional costs one cannot expect that all weight vectors in $\Delta_m$ will be attained, only those residing in the Hall polytope. Let us briefly describe the main steps for proving the existence of a transport map of some measure $\mu$ to a discrete measure $\nu$ with weight vector in the interior of the Hall polytope (i.e. Theorem \ref{thm:discrete-transport}). Fixing a measure $\mu$ and an $m$-tuple $(u_i)_{i=1}^m$, the construction in 
Lemma \ref{fact:c-subgradient-of-discrete-inf} gives rise to a mapping $H$ from 
the $(m-1)$-dimensional simplex $\Delta_m$ onto the set of   `weight vectors' $\alpha=(\alpha_i)_{i=1}^m$ of the measure $\nu$ to which $\mu$ can be transported. 
We will show that $H$ is a surjection from the interior of the simplex onto the interior of the relevant Hall polytope. To this end we define and analyze Hall polytopes, and in particular construct, under some assumptions, a continuous map $R$ from the boundary of the polytope to the boundary of the simplex, which respects certain constraints connected with the face structure of the polytope.  We use a variant of Brouwer's fixed point theorem for the composition   $R\circ p \circ H$, where $p$ is a radial projection from some point in the polytope, to obtain the surjectivity.

\subsection{Structure of the Hall Polytope}

		We introduce the following notation: For $I \subset [m]$, $A\subset \RR^{|I|}$ and $B\subset \RR^{m-|I|}$, we denote by $A\times_I B$ points in $\RR^m$ with $I$-coordinates in $A$ and $I^c$-coordinates in $B$. For a measure $\mu$ and a set $A\subset X$ we denote by $\mu|_A$ the measure that is equal to $\mu$ on $A$ and zero on $A^c$.

	Hall polytopes have faces only in specific pre-determined directions.  (Their faces'  normal cones are spanned by $\{0,1\}$-vectors in $\RR^m$, projected onto the span of the polytope which is $(m-1)$ -dimensional.)
	As we shall see in Proposition \ref{prop:splitting}, each of these faces has a product structure, of which each component is a Hall polytope itself.  
		
\begin{prop}\label{prop:splitting}
	Let $P = P((u_i)_{i=1}^m,\mu)$ be the Hall polytope associated with some $m$-tuple 
	$(u_i)_{i=1}^m\subset Y$ and a 
	probability measure $\mu \in {\mathcal P}(X)$ supported on $\{x\in X: \exists i\in [m]\  c(x,u_i)<\infty\}$. 
	Then for each $I \subset [m]$, the face of $P$ given by
	\begin{align}\label{eq:def-F_I}
	    F_I = \{ \alpha \in P: \sum_I \alpha_i =  \mu(A_I)\},
	\end{align} 
	admits a splitting $ F_I = \mu(A_I)P_I \times_I \mu(A_I^c)\hat{P}_I $ where $P_I$ is the Hall polytope associated with the measure $\frac{1}{\mu(A_I)}\mu|_{A_I}$ and the vectors $\{u_i\}_{i\in I}$, and $ \hat{P}_I $ is  the Hall polytope associated with the measure $\frac{1}{\mu(A_I^c)}\mu|_{A_I^c}$ and the vectors $\{u_i\}_{i\in [m]\setminus I}$. In particular, in case $\mu(A_I)=0$, we have $P_I = \{0\}_I$, and in case $\mu(A_I^c)=0$, $\hat{P}_I = \{0\}_{[m]\setminus I}$. 
		\end{prop}

\begin{proof}
			Let $\alpha \in F_I$, then by the definition of $F_I$ we have $\sum_{i\in I} \alpha_i = \mu(A_I)$ and thus $\alpha|_I \in \mu(A_I)\Delta_{|I|}$. Furthermore, for every $J\subset I$ it still holds that $\sum_{i\in J} \alpha_i \le \mu(A_J)$ and as $A_J \subset A_I$ we also get that 
			$\sum_{i\in J} \alpha_i \le \mu|_{A_I}(A_J)$. Recall that  $P_I$ is the Hall polytope associated with $(u_i)_{i\in I}$  and $\mu(A_I)^{-1}\mu|_{A_I}$,  so re-normalizing the previous inequalities by $\mu(A_I)$ we see that the vector  
			$\alpha|_I \in \mu(A_I)P_I$, as claimed. 
			
			Similarly in the $I^c$ coordinates, $\alpha \in F_I$ satisfies $\sum_{i\in [m]\setminus I} \alpha_i = \mu(A_I^c)$. 
			To show $\alpha|_{[m]\setminus I}\in \mu(A_I^c) \hat{P}_I$ we need to check that for every $K\subset [m]\setminus I$ we have $\sum_{i\in K}\alpha_i \le \mu(A_K \cap A_I^c)$. To this end consider the new subset of $[m]$ given by $J = I \cup K$. By the assumptions, 
			\[ \sum_{i\in J} \alpha_i \le \mu (A_J) = \mu \big( \bigcup_{i\in I} \{x:\  c(x,u_i)<\infty \} \cup \bigcup_{i\in K} \{x:\  c(x,u_i)<\infty \} \big). \] 
			Since the first of these unions is in fact all of $A_I$, we may rewrite the inequality as 
			\[ \sum_{i\in J} \alpha_i \le  \mu (A_I) + \mu (A_I^c \cap \bigcup_{i\in K} \{x:\  c(x,u_i)<\infty\} ). \] 
			The sum on the left hand side is simply $\sum_{i\in I}\alpha_i + \sum_{i\in K} \alpha_i  = \mu(A_I)+ \sum_{i\in K} \alpha_i$, since we have assumed $\alpha \in F_I$. Plugging into the inequality and canceling, we see 
			\[ \sum_{i\in K} \alpha_i \le   \mu (A_I^c \cap \bigcup_{i\in K} \{x:\  c(x,u_i)<\infty \} ),  \]
			as claimed. 
			
			We have thus shown, so far, that $F_I \subset\mu(A_I)P_I \times_I \mu(A_I^c)\hat{P}_I $. For the opposite direction, assume we are given some point $\alpha \in \mu(A_I)P_I \times_I \mu(A_I^c)\hat{P}_I $, and we want to show that it belongs to $F_I$. Clearly, using that if $K\subset J$ then $A_K \subset A_J$,  we have for any $J\subset [m]$ that	  
			\begin{eqnarray*}\sum_{i\in J} \alpha_i &=& \sum_{i\in J\cap I} \alpha_i  + \sum_{i\in J\cap ([m]\setminus I)} \alpha_i \le \mu|_{A_I} (A_{J\cap I}) + \mu|_{A_I^c} (A_{J\cap ([m]\setminus I)})\\
				&\le& \mu|_{A_I} (A_{J}) + \mu|_{A_I^c} (A_{J})  = \mu(A_J). \end{eqnarray*}
			
			This completes the second part of the proof. 
		\end{proof}

		We will discuss the facial structure of the polytope, and make use of the following simple observation. 
		
		\begin{lem}\label{lem:bdofFI}
			Under the conditions and notations of Lemma \ref{prop:splitting}, for any $I\subset [m]$, the part of the boundary of $F_I$ given by $\mu(A_I)\partial P_I 
	\times_I \mu(A_I^c)\hat{P}_I$ is a subset of $\cup_{ J \subsetneq I}F_J$.		
		\end{lem}
		
		\begin{proof}
			The boundary of $P_I$ consists of points whose $I^{th}$ coordinates add up to one, and for some $J\subsetneq I$ one of the inequalities defining the  
			Hall polytope associated with $\frac{1}{\mu(A_I)}\mu|_{A_I}$ and $(u_i)_{i\in I}$ is an equality. In other words, if $\alpha \in \mu(A_I)\partial P_I 
			\times_I \mu(A_I^c)\hat{P}_I$ 			
			there is some $J\subsetneq I$ such that  $\sum_{i\in J}\alpha_i =  \mu(A_J)$, which means $\alpha \in F_J$, as claimed. 
		\end{proof}

		\subsection{Non-degenerate polytopes}\label{subsec:non-degenerate}
		In this subsection we continue analyzing properties of Hall polytopes, under an additional assumption on $\mu$ and $(\alpha_i)_{i=1}^m$ which will imply that all of the Hall polytopes' faces $F_I$ (defined in \eqref{eq:def-F_I}) are `full dimensional' in the $I$ coordinates, i.e.  that in the splitting described in Proposition \ref{prop:splitting}, the polytope $P_I$ is $|I|-1$ dimensional.

		\begin{definition}\label{def:non-degenerate-hall}
	Let $X,Y$ be measure spaces and let  $c:X\times Y \to (-\infty, \infty]$ be a measurable cost function. Given $(u_i)_{i=1}^m\subset Y$, and a %
	probability measure $\mu$  which is supported on $\{x\in X: \exists i\in [m]\  c(x,u_i)<\infty\}$, we say that $\mu$  is \emph{\bf {non-degenerate with respect to  $(u_i)_{i=1}^m$}}  if 
			for every $1\le i<j\le m$ it holds that
			\[ \mu \left(\{x: c(x,u_i)<\infty\} \cap \{x: c(x,u_j)<\infty\}\right)>0.\] 
		\end{definition}

	\begin{figure}[H]
	\begin{subfigure}{.33\textwidth}
		\centering
		\includegraphics[width=.8\linewidth]{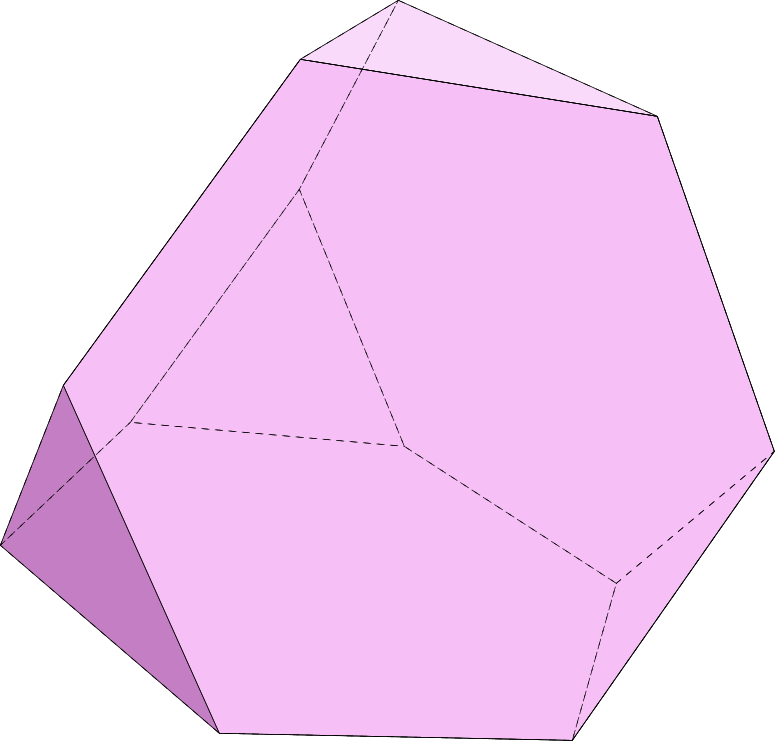}
		\caption{Non-degenerate}
		\label{fig:sfig1}
	\end{subfigure}%
	\begin{subfigure}{.33\textwidth}
		\centering
		\includegraphics[width=.8\linewidth]{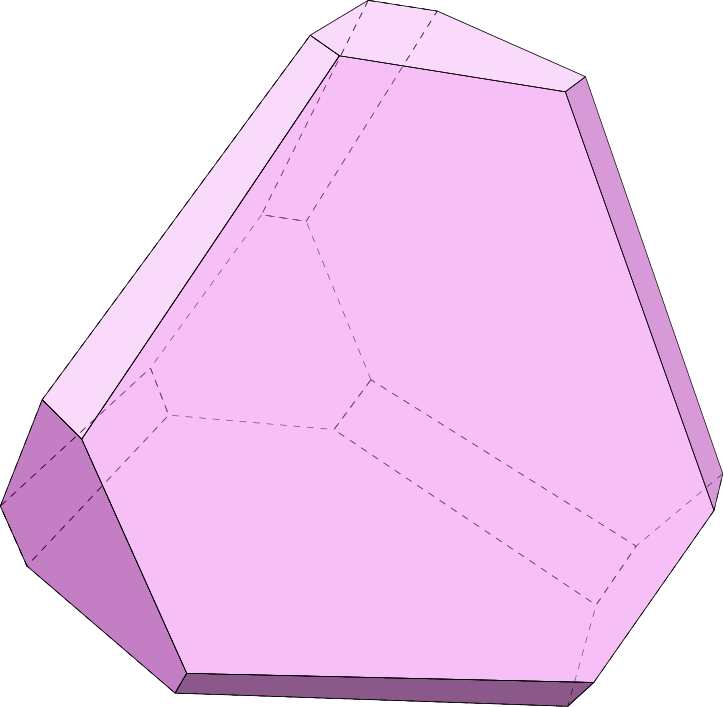}
		\caption{Non-degenerate}
		\label{fig:sfig2}
	\end{subfigure}
	\begin{subfigure}{.33\textwidth}
		\centering
		\includegraphics[width=.7\linewidth]{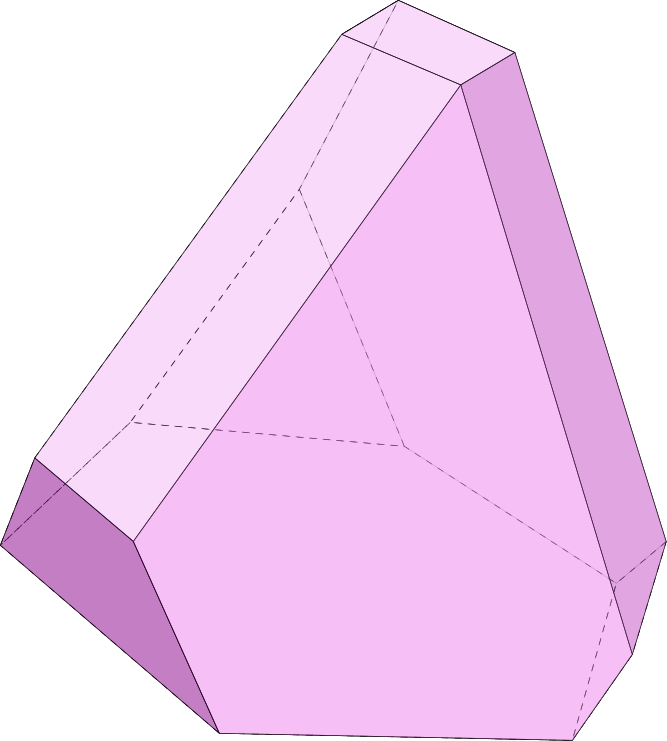}
		\caption{Degenerate}
		\label{fig:sfig3}
	\end{subfigure}
	\caption{Examples of 3-dimensional Hall polytopes}
	\label{fig:fig}
\end{figure}

		\begin{prop}\label{prop:Pfulldim}
			Given a %
			probability measure $\mu\in {\mathcal P}(X)$, which is non-degenerate with respect to $(u_i)_{i=1}^m \subset Y$, the Hall polytope $P = P((u_i)_{i=1}^m,\mu)$ satisfies that its dimension (meaning the dimension of its affine hull) is $\dim(P) = m-1$. 
		\end{prop}

		\begin{proof}
			We shall prove this fact using induction on $m$. For $m = 1$ this is clearly true since 
			the polytope $P$ consists of one point $\alpha = 1$, that is, has dimension $0$. 
			Assume that the claim is true for $(m-1)$-tuples. Then, for $m$ and a given set of vectors $(u_i)_{i=1}^m \subset Y$, we know by Proposition \ref{prop:splitting} that $P$ has faces $F_I$ where $I\subset [m]$, each of the form $F_I = \mu(A_I)P_I \times_I \mu(A_I^c)\hat{P}_I$. Let $I_1 = [m-1]$ and $I_2 = \{m\}$. The set $[m-1]$ still satisfies, along with $\mu|_{I}$, the conditions of the proposition, so by the inductive assumption $P_1 = P_{I_1}$ is a polytope of full dimension, that is, of dimension $m-2$.
			
			It remains to show that $F_2 = F_{I_2}$ does not lie within the affine hull of $F_1$, and hence $P$ has dimension at least $m-1$ (and of course it cannot have a higher dimension, as it is a subset of $\Delta_m$).
			Note that the affine hull of  $F_1$ is characterized by the equality $\sum_{i=1}^{m-1} \alpha_i = \mu (A_{[m-1]})$, which equivalently can be written as $ \alpha_{m} = 1-\mu (A_{[m-1]})$. The facet $F_2$ satisfies $\alpha_m = \mu(A_{\{m\}})$.
			Assuming towards a contradiction that these two facets do intersect, we would need to have $\mu(A_{[m-1]}) + \mu (A_{\{m\}}) = 1$. Recall that 
			\[A_{\{m\}} = \{x\in X:\ c(x,u_m) < \infty\},\ A_{[m-1]} = \{x\in X:\ \exists i \in [m-1]\ \ c(x,u_i) < \infty \}\]
			and that by the non-degeneracy of $\mu$ it holds that $\mu(A_{\{m\}} \cap A_{[m-1]})>0$. Additionally,  $\mu(A_{\{m\}} \cup A_{[m-1]}) =1$, and so
			\[\mu(A_{\{m\}} \cup A_{[m-1]}) = \mu(A_{\{m\}}) + \mu(A_{[m-1]}) - \mu(A_{\{m\}} \cap A_{[m-1]})\]
			implies that $\mu(A_{\{m\}} \cap A_{[m-1]})=0$ (So in particular $\mu(A_{\{m\}} \cap A_{\{1\}})=0$), contradicting the assumption that $\mu$ is non-degenerate.
		\end{proof}

		\begin{cor}\label{cor:face-structure}
			Given a	probability measure $\mu\in {\mathcal P}(X)$ which is %
			non-degenerate with respect to $(u_i)_{i=1}^m \subset Y$, the Hall polytope $P = P((u_i)_{i=1}^m, \mu)$ satisfies that each face $F_I$ admits a splitting $F_I = P_I \times \hat{P}_I$ such that $\dim(P_I) = |I|-1$. 	
		\end{cor}

		\begin{proof}
			The fact that $F_I$ has such a splitting was proven already in Proposition \ref{prop:splitting}, with $P_I$ being the Hall polytope of the normalized restriction of the measure $\mu$ to $A_I$. $\mu_I$ satisfies, together with the subset $(u_i)_{i \in I}$, conditions of Proposition \ref{prop:Pfulldim}, namely that it is non-degenerate with respect to $(u_i)_{i\in I}$ as one may easily check that $\mu|_{A_I}$ is non-degenerate with respect to $(u_i)_{i\in I}$.
			Therefore, $P_I$ is full dimensional, as claimed. 
		\end{proof}

	 	Furthermore, in this case the associated polytope satisfies a ``good" face-intersection structure, explained in the next two propositions.  
		
\begin{prop}\label{prop:intersection-style}
Given a	probability measure $\mu\in {\mathcal P}(X)$ supported on $\{x\in X: \exists i\in [m]\  c(x,u_i)<\infty\}$ which is %
non-degenerate with respect to $(u_i)_{i=1}^m \subset Y$, let $P = P((u_i)_{i=1}^m, \mu)$ be the associated Hall polytope.
	Given $I,J\subset [m]$, the intersection $F_I \cap F_J$ is a subset of $F_{I\cap J}$ (we let $F_\emptyset = P$, so that if $I\cap J = \emptyset$ the claim is trivial).
		\end{prop}
		
		\begin{proof}
			Let $\alpha \in F_I \cap F_J$, so that $\sum_{i\in I} \alpha_i = \mu(A_I)$ and 
			$\sum_{i\in [m]\setminus I} \alpha_i = \mu(A_I^c)$, as well as   $\sum_{i\in J} \alpha_i = \mu(A_J)$ and 
			$\sum_{i\in [m]\setminus J} \alpha_i = \mu(A_J^c)$. Consider the following equation 
			\begin{eqnarray*}
				\sum_{i\in I\cap J} \alpha_i + \sum_{i\in I \cup J} \alpha_i &=& 	\sum_{i\in I} \alpha_i + \sum_{i\in J} \alpha_i  = \mu (A_I) + \mu(A_J)  \\ &=& \mu(A_I \cap A_J) + \mu(A_I \cup A_J) \ge \mu(A_{I\cap J})+ \mu(A_{I \cup J}),
			\end{eqnarray*}
			where the final inequality follows from the inclusion $A_{I \cap J} \subset A_I \cap A_J$.
			
			Pairing this with the fact that each of the extreme terms satisfies that
			\[\sum_{i\in I\cap J} \alpha_i \le \mu(A_{I\cap J})\quad {\rm and}\quad
			\sum_{i\in I \cup J} \alpha_i\le \mu(A_{I \cup J}),
			\]
			we conclude that both of these inequalities are in fact equalities, which implies that 
			\[\sum_{i\in I\cap J} \alpha_i = \mu(A_{I\cap J}),
			\]
			so that $\alpha$ belongs to the facet $F_{I\cap J}$. 
		\end{proof}
		
		In fact, if $\mu$ is non-degenerate and $I\cap J = \emptyset$, we know much more.
		
\begin{prop}\label{prop:nointersection}
Given a	probability measure $\mu\in {\mathcal P}(X)$ supported on $\{x\in X: \exists i\in [m]\  c(x,u_i)<\infty\}$ which is 
non-degenerate with respect to $(u_i)_{i=1}^m \subset Y$, let $P = P((u_i)_{i=1}^m, \mu)$ be the associated Hall polytope.
Then given $I,J\subset [m]$, which are disjoint,  the  faces $F_I $ and $ F_J$ do not intersect. 
		\end{prop}

		\begin{proof}
			
			By non-degeneracy of $P=P((u_i)_{i=1}^m, \mu)$, we know that for any $I$, the face $F_I = \mu(A_I)P_I \times_I \mu(A_I^c)\hat{P}_I$ satisfies that $\dim (P_I) = |I|-1$.
			Assume $|I|= k_1,|J| = k_2$, and, towards a contradiction, that the intersection $F_I \cap F_J$ is non-empty. Denoting $\beta_I=\mu(A_I)$ and $\beta_J=\mu(A_J)$, every point $\alpha$ in the intersection must satisfy 
			$\sum_{i\in I} \alpha_i = \beta_I$ and $\sum_{i\in J}\alpha_i = \beta_J$. Letting $K = I \cup J$, $\beta_K = \mu(A_K)$, by the fact that $I$ and $J$ are disjoint, we see that $\sum_{i\in K} \alpha_i = \beta_I+ \beta_J$, and since $\alpha$ is a point in $P$, $\beta_I+ \beta_J\le \beta_K$. However, using again that $I\cap J = \emptyset$, we know that all points $\alpha\in P$ also satisfy $\sum_{i\in K} \alpha_i \le \beta_I+ \beta_J$, and since by Proposition \ref{prop:Pfulldim} $F_K$ is non-empty (it is full dimensional in its $I$ coordinates), there exists some $\alpha \in F_K$ such that the equality $\sum_{i\in K} \alpha_i = \beta_K$ is satisfied. This implies $\beta_K= \beta_I+ \beta_J$. We conclude that $F_I \cap F_J =F_K$. Indeed, $F_K \subset F_I \cap F_J$, since each such $\alpha \in P$ satisfies $\sum_{i\in I} \alpha_i \le \beta_I$ and  $\sum_{i\in J} \alpha_i \le \beta_J$, and for points in $F_K$ an equality must be attained in both inequalities. The reverse inclusion $F_I \cap F_J \subset F_K$ is clear.
			
			However, by Proposition \ref{prop:Pfulldim} the dimension of $P_I$ is $k_1-1$ and the dimension of $P_J$ is $k_2 - 1$. Recalling that $K= I \cup J$, we see that the dimension of $P_K$ is at most  $k_1 - 1 +k_2 - 1 < k_1 + k_2 -1=|K|-1$, which contradicts the non-degeneracy assumption on $P$, and implies that the intersection must be empty.
			
		\end{proof}

\subsection{Mapping the Hall polytope to the simplex} \label{subsec:R-map}	
			
In this subsection we make one final preparation, and show that for any Hall polytope $P$, associated with a non-degenerate measure and some   $m$-tuple, there exists a special mapping $R$ from $\partial P$ to $\partial\Delta_m$  such that $F_I$ is mapped to $\pa_I \Delta_m$, and on $F_I$, the map only depends on the $I$ coordinates of a point. %
			
			Let us explain the notation. The relative boundary of the simplex (its boundary in the affine space $\{\alpha \in \RR^m:\ \sum_{i=1}^m \alpha_i = 1\}$) will be denoted by $\pa \Delta_m$, and the $I^{\rm th}$ component of this boundary is the lower dimensional simplex defined by
			\[\pa_I \Delta_m=\{\alpha \in \Delta_m: \ \sum_{i\in I} \alpha_i = 1\}.\]
			Additionally, for $I\subset [m]$ we say that `a point $x\in \RR^m$ has $I^{th}$ coordinates $y\in \RR^{|I|}$' if the restriction of $x$ to its coordinates indexed by $I$ is equal to $y$.
			
			\begin{prop}\label{prop:mapping-P-to-simplex}
Given a	probability measure $\mu\in {\mathcal P}(X)$ supported on $\{x\in X: \exists i\in [m]\  c(x,u_i)<\infty\}$ which is 
non-degenerate with respect to $(u_i)_{i=1}^m \subset Y$, let $P = P((u_i)_{i=1}^m, \mu)$ be the associated Hall polytope.
		Then there exists 	 a continuous mapping $R: \partial P \to \partial \Delta_m$ such that
				$\partial_I P : = F_I =\mu(A_I)P_I \times_I \mu(A_I^c)\hat{P}_I$ is mapped to $\partial_I \Delta_m$  with
				\begin{equation}\label{eq:R-mapping-requirement}
				R(y,z)=R(y,z')
				\end{equation}
				for $y\in \mu(A_I)P_I, z, z'\in \mu(A_I^c) \hat{P}_I$, that is $R(x)=R(x')$ if $x|_{I^c} = x'|_{I^c}$.
			\end{prop}
			
	\begin{proof}
	The construction of $R$ is recursive. 
	We define the map first only on faces  $F_I$ with $|I| = 1$. We then assume it has been defined on faces $F_I$ with $|I|<k$ and define it on $F_I$ with $|I| = k$. At each step we  make sure the map we construct is well defined and continuous on its domain. 
				
	We denote the center of mass of the face $\partial_I \Delta_m$ by $q_I$ and the center of mass of the polytope $\mu(A_I)P_I$ (the $I$-th component of $F_I$) by $p_I$. We will ensure, within the proof, that all points in $F_I$ with $I^{th}$ coordinates $p_I$  are mapped to $q_I$, and that in general the map on a face $F_I$ depends only on the $I^{th}$ coordinates of the point. 
				
	The basis for the construction are thus faces $F_I$ of $P$ with $|I| = 1$. These are mapped to the vertices of the simplex $\Delta_m$, namely $F_{\{i\}} \mapsto \partial_{\{i\}} \Delta_m$ (where $\partial_{\{i\}} \Delta_m= e_i$). As these faces are disjoint by Proposition \ref{prop:intersection-style}, and the map $R$ is constant on each face, we conclude that it is continuous.

For the   induction step, assume we have defined $R$ on all faces $F_J$ with $|J|<k$. Let $F_I$ be a face of $P$ with $|I|=k$. Since $F_I = \mu(A_I)P_I \times_I \mu(A_I^c)\hat{P}_I$  by Proposition \ref{prop:splitting}, we have that 
\[\pa F_I = \left( \mu(A_I){\rm relint} (P_{I}) \times \mu(A_I^c) \partial \hat{P}_{I} \right) \cup \left( \mu(A_I) \partial   P_{I} \times  \mu(A_I^c) \hat{P}_{I} \right). 
				\]
Since $R$ is already defined, by assumption, on all faces $F_J$ for $I\neq J\subset I$, and since  $ \mu(A_I) \partial    P_I \times  \mu(A_I^c)  \hat{P}_I \subset \bigcup_{J\subsetneqq I} F_J$ by Lemma \ref{lem:bdofFI}, the map $R$ is already defined on this first component of the boundary. Furthermore, again by assumption, it is defined in such a way that the image of $F_J$ is the simplex $\partial_J \Delta_m$, and that on $F_J$ the map $R$ only depends on the $J^{th}$ coordinates of the point. Note that $\bigcup_{J\subsetneqq I} \pa_J\Delta_m $ is precisely the boundary of $\partial_I \Delta_m$. So, we essentially are given a continuous mapping $R$ from the boundary of $P_I$ to the boundary of $\partial_I\Delta_m$.  We extend it by first imposing $R(p_I, z) = q_I$  for the specified points $p_I$ and $q_I$ (note that $(p_I, z)\in F_I$ lies in $\mu(A_I){\rm relint} (P_{I})$, as $P_I$ is full, i.e. $I-1$, dimensional), and then extending $R$ radially for points with $I^{th}$ coordinates in $\mu(A_I){\rm relint} (P_{I})$. The resulting map $R$ is now defined on all of $F_I$. We do this for all index sets $I$ of size $k$. The resulting map is well defined, since by Propositions \ref{prop:intersection-style} and \ref{prop:nointersection} the  intersections of the faces $\{F_I\}_{|I|=k}$ are included in faces $F_J$ with $|J|<k$. By construction $R$ is a continuous mapping that sends $F_I$ to $\pa_I \Delta_m$ and, on $F_I$, depends only on the $I^{th}$ coordinates.

			\end{proof}

	As we described in Subsection \ref{subsection: building transport maps}, the main idea of the proof of  Theorem \ref{thm:discrete-transport} is to show surjectivity of a map taking a potential function $\varphi$ (indexed by some variables $(t_i)\in \Delta_m$) to the weight vector $\alpha$ of the measure $\nu$ to which the $c$-subgradient $\partial^c\varphi$ maps $\mu$. We present this formally in the next subsection, where we define and analyse this map.

\subsection{Mapping the simplex to the Hall polytope} \label{subsec:H-map}

Having fixed some $m$-tuple $(u_i)_{i=1}^m\subset Y$ and a 
probability measure $\mu \in {\mathcal P}(X)$ supported on $\{x\in X: \exists i\in [m]\  c(x,u_i)<\infty\}$ and $c$-regular, 
we define a map on the interior of $\Delta_m$, and then extend it (using converging subsequences) to a set-valued map on the boundary.

More precisely, for $t=(t_i)_{i=1}^m \in {\rm int} (\Delta_m)$, define
			\begin{equation}\label{eq:def-H}
			H_{(u_i)_{i=1}^m}^\mu (t)= \alpha
			\end{equation}
			with $\alpha \in \Delta_m$ given by
			\[\alpha_i = \mu\left( \left\{ x\in X:\ \arg\min_{1\leq j\leq m} c(x,u_j) -\ln(t_j) = i \right\} \right).\] 
			For $t\in \pa \Delta_m$, we let $H_{(u_i)_{i=1}^m}^\mu (t)$ be the closure of the function in the usual sense, namely the set of all limit points $\lim H_{(u_i)_{i=1}^m}^\mu (t^{(k)})$ as $t^{(k)} \to t$ and $t^{(k)}\in{\rm int} (\Delta_m)$.  When $\mu$ and $(u_i)_{i=1}^m$ are fixed in advance, we denote $H = H_{(u_i)_{i=1}^m}^\mu$.
By Lemma \ref{fact:c-subgradient-of-discrete-inf} there is a transport map from $\mu$ to $\nu = 	 \sum_i \alpha_i \ind_{u_i}$ when $\alpha = H(t)$ for $t\in {\rm int} (\Delta_m)$, and moreover the transport map's graph is included in the $c$-subgradient of the function 	
$\min_{1\leq j\leq m} c(x,u_j) -\ln(t_j)$. In particular, the image of $H$ is inside the 
$P = P((u_i)_{i=1}^m,\mu)$, the associated Hall polytope.

 Our first claim regards the continuity of $H$. 
	\begin{prop}\label{prop:H-continuous} Let $X$ and $Y$ be measure spaces,  $c:X\times Y \to (-\infty, \infty]$ measurable, let $(u_i)_{i=1}^m\subset Y$,  and let 
  $\mu \in {\mathcal P}(X)$ be $c$-regular and supported on $\{x\in X: \exists i\in [m]\  c(x,u_i)<\infty\}$.  Then, the function $H = H_{(u_i)_{i=1}^m}^\mu:\Delta_m \to P((u_i)_{i=1}^m, \mu)$ is well defined and continuous on ${\rm int}(\Delta_m)$.%
			\end{prop}

\begin{proof} First note that the function $H$ is well defined as $\mu$ is $c$-regular, and the subsets 
\[U_i := \left\{ x\in X:\ \arg\min_{1\leq j\leq m} c(x,u_j) -\ln(t_j) = i \right\}\]  
form a measurable partition of $X$ (the intersections are of measure $0$, as well as the set where the minimum is $+\infty$), as in Lemma \ref{fact:c-subgradient-of-discrete-inf}.

To show that $H$ is continuous on ${\rm int}(\Delta_m)$, let $t \in {\rm int}(\Delta_m)$,  $\varepsilon >0$ be fixed.  We will show that then there exists $\delta>0$ such that for all $t' \in \Delta_m$ with $\|t'-t\|_2<\delta$, we have $\|H(t')-H(t)\|_2 < \varepsilon$. To see this, note that the $i^{th}$ coordinate of the difference is given by $\mu(U_i) - \mu(V_i)$, where 
	\[V_i := \left\{ x\in X:\ \arg\min_{1\leq j\leq m} c(x,u_j) -\ln(t'_j) = i \right\}.\]
Clearly, this difference is bounded (in absolute value) by $\mu(U_i \triangle V_i)$, where $\triangle$ denotes the symmetric difference of the two sets. To estimate the measure of the symmetric difference, when $t$ and $t'$ are close, we use the following sets, which converge to measure $0$ sets as $k\to \infty$.

Define for $i\in [m], k\in {\mathbb N}$ 
	\[U_i^k = \left\{ x\in U_i: \exists j\neq i, \ c(x,u_j)-\ln(t_j) - c(x,u_i) + \ln(t_i) \le \frac{1}{k} \right\}. \] 
	
Note that $U_i^{k+1} \subset U_i^k $, and as $\mu$ is  finite  $\mu(U_i^k) \searrow  \mu(\lim_k U_i^k)$. Moreover,  $\mu(\lim_k U_i^k) = 0$ since by $c$-regularity of $\mu$ the limit set $\cup_{i\neq j}\{x\in X:\  c(x,u_j)-\ln(t_j) - c(x,u_i) + \ln(t_i) = 0  \}$ has zero measure. In particular, for every $i\in [m]$ there exists some $k_i$ such that for all $k>k_i$, we have $\mu(U_i^{k})<\varepsilon/m$.  Denote $k_0 = \max_i k_i$, and note that for any $k \ge k_0$ we have $\mu(\cup_{j=1}^m U_j^{k})<\varepsilon$ (and in particular for $k = k_0$).

We next claim that there exists $\delta$ such that if $t'$ is such that $|t'_i-t_i|<\delta$ for all $i$, we have 	that			
				 $U_i \triangle V_i \subset \cup_{j=1}^m U_j^{k_0}$, which completes the proof.  Indeed, we will choose $\delta$ such that if $|t-t'|<\delta$ then $t_le^{-1/2k_0}\le t_l'\le t_le^{1/2k_0}$ for every $l$.

  First consider the case $x\in U_i\setminus V_i$, and note that then there exists $1\le l\le m$ such that $x\in V_l$ (since the sets $(V_l)_{l=1}^m$ are a partition of $X$) and hence for all $1\le j\le m$ we have
\[ c(x,u_l)-\ln(t'_l) - c(x,u_j) + \ln(t'_j)\le 0.\]	
By taking $j=i$ and by the choice of $\delta$, 
		\[ 	c(x,u_l)-\ln(t_l ) - c(x,u_i) + \ln(t_i ) \le 	\frac{1}{k_0},\]
which yields that $x\in U_i^{k_0}$. Similarly, in the case where $x\in V_i\setminus U_i$, there exists $1\le l\le m$ such that $x\in U_l$ and since $x\in V_i$ we have that for all $1\le j\le m$ 
\[c(x,u_i)-\ln(t_i')-c(x,u_j)+\ln(t_j') \le 0.
\]
By taking $j=l$ and using the assumption on $\delta$, this yields
\[c(x,u_i)-\ln(t_i )-c(x,u_l)+\ln(t_l ) \le 	\frac{1}{k_0}, 
\]
which implies $x\in U_l^{k_0}$, and in particular, in both cases, $x\in\cup_{j=1}^m U_j^{k_0}$. Since the parameters were chosen so that the measure of this set is at most $\eps$, we conclude that $\mu(U_i \triangle V_i)<\eps$, so long as $|t'-t|<\delta$, which completes the proof.  
			\end{proof}

A main feature of the map $H_{(u_i)_{i=1}^m}^\mu$ is that it respects the product structure on the faces of $\Delta_m$. More precisely, when applied to a point on a face $\partial_I\Delta_m$, the map is usually set-valued. The set which such a point is mapped to, however, has a specified $I^{th}$-coordinate (given by another map of such form, associated with a different measure), and the $I^c$-coordinates of points in the image span a full Hall polytope of another associated measure -- exactly the one given in the face splitting discussed in Proposition \ref{prop:splitting}. This is formally described in the next proposition.
			
\begin{prop}\label{prop:description-of-H}
Under the assumptions of Proposition \ref{prop:H-continuous}, consider some subset $I\subset [m]$ and let $t = t_I \times _I 0_{I^c}$ be a vector  with positive  $I^{th}$-coordinates.
	Let
	\[\mu_I = \frac{1}{\mu(A_I)}\mu\big|_{A_I}, \, \mu_I^c=\frac{1}{\mu(A_I^c)}\mu\big|_{A_I^c}.\]
	Then,
	\[H_{(u_i)_{i=1}^m}^\mu (t_1,\dots,t_m) = \mu(A_I)H_{(u_i)_{i\in I}}^{\mu_I}\left(t_i \right)_{i\in I} \times_I \mu(A_I^c)H_{(u_j)_{j\in I^c}}^{\mu_I^c}(\Delta_{m-|I|}).\]
	In particular, $H_{(u_i)_{i=1}^m}^\mu$ maps the face $\partial_I\Delta_m$ to the face $F_I$ of the Hall polytope  $P((u_i)_{i=1}^m, \mu)$.
			\end{prop}

\begin{proof}[Proof of Proposition \ref{prop:description-of-H}]
				
				We   show  a two-way inclusion.\\				
For the direction $\supseteq$ take a point $(\alpha_1,\dots, \alpha_m)$ in the right hand side, which is of the form
\[\mu(A_I)H_{(u_i)_{i\in I}}^{\mu_I}(t_i)_{i\in I} \times_I \mu(A_I^c)H_{(u_j)_{j\notin I}}^{\mu_I^c}(s_j)_{j\notin I}\]
	where $\sum_{i\in I} t_i = 1$ and $\sum_{j \notin I} s_j =1$. For $\delta \in (0,1)$ define   ${t}^\delta \in\Delta_m$ in the following way
	\begin{equation*}
	{t}^\delta_i = \begin{cases}
	(1-\delta) t_i & i\in I\\
	\delta s_i & i\notin I
	\end{cases}
	\end{equation*} 
	Clearly, $t^\delta \to t$  as $\delta \to 0$, thus, by continuity of $H$ (Proposition \ref{prop:H-continuous}), it suffices to show that
	\[(\alpha_1,\dots, \alpha_m) \in \lim_{\delta \to 0} H_{(u_i)_{i=1}^m}^{\mu}({t}^\delta_i)_{i=1}^m.\]
    We will show that for every $\varepsilon >0$ there exists some $\delta_0$ such that for every $\delta< \delta_0$  we have
				\[||(\alpha_1,\dots, \alpha_m) -  H_{(u_i)_{i=1}^m}^{\mu}({t}^\delta)|| \leq \varepsilon.\]
				Denote
				$(\beta_1,\dots, \beta_m) = H_{(u_i)_{i=1}^m}^{\mu}({t}^\delta)$. 
				Let us reinterpret $\beta_i$,
				\begin{eqnarray*} 
				\beta_i &=& \mu (  \{ x \in X :  \arg\min_{1\leq k \leq m}   c(x,u_k) - \ln({t}^\delta_k) = i   \}  )\\
				&=& \mu (  \{ x \in A_I^c : \arg\min_{1\leq k \leq m}   c(x,u_k) -\ln({t}^\delta_k)= i  \}  )\\&& +  \mu (  \{ x\in A_I:\ \arg\min_{1\leq k \leq m}   c(x,u_k) -\ln({t}^\delta_k) = i  \}  ). 
				\end{eqnarray*}
				On $A_I^c$ the minimum is attained for $k\notin I$,  hence
				\begin{align*} 
				\beta_i& = \mu(A_I^c)\mu_I^c ( \{ x\in X :\ \arg\min_{1\leq k \leq m}   c(x,u_k) -\ln(\delta s_k ) = i  \} ) \\
				&\quad +  \mu (  \{ x\in A_I:\ \arg\min_{1\leq k \leq m}   c(x,u_k)-\ln({t}^\delta_k) = i  \} ) \\
				&= \mu(A_I^c)\mu_I^c (  \{ x\in X :\ \arg\min_{1\leq k \leq m}   c(x,u_k) -\ln( s_k)= i  \}  ) \\
				&\quad +  \mu (  \{ x\in A_I:\ \arg\min_{1\leq k \leq m}  c(x,u_k)-\ln({t}^\delta_k ) = i  \}  ). 
				\end{align*}

				Observe that the first summand is by definition equal to
		\begin{equation}\label{eq:key-lemma-observation}
				\mu(A_I^c)\mu_I^c (  \{ x\in X :\ \arg\min_{1\leq k \leq m}   c(x,u_k)-\ln(s_k ) = i  \}  ) = \begin{cases}
				0 & i\in I\\
				\alpha_i & i \notin I
				\end{cases}.
				\end{equation}

	We first deal with the case $i\notin I$, in which \eqref{eq:key-lemma-observation} gives that 
				\[\beta_i = \alpha_i + \mu (  \{ x\in A_I:\ \arg\min_{1\leq k \leq m}   c(x,u_k) -\ln({t}^\delta_k) = i  \}  ).\]			
	 For $x\in A_I$, $\arg\min_{1\leq k \leq m}   c(x,u_k) -\ln({t}^\delta_k) = i$ means, in particular, that  
		\[c(u_j,x)-\ln((1-\delta)t_j) > c(u_i,x)-\ln(\delta s_i)\]
		for all $j\in I$. Since $s_i$ and $t_j$ are fixed, we can clearly find $\delta_0>0$ such that for any $\delta<\delta_0$ the measure  of such $x$'s is arbitrarily small. Thus, we choose $\delta_0$ (depending on $s$ and $t$) such that  $\alpha_i \le \beta_i \le \alpha_i + \varepsilon/m$.

For the case $i \in I$,
	            \begin{align*}
				\beta_i &= \mu (  \{ x\in A_I:\ \arg\min_{1\leq k \leq m}   c(x,u_k) -\ln({t}^\delta_k) = i \} ) \\
				&\leq \mu (  \{ x\in A_I:\ \arg\min_{k \in I}  c(x,u_k) -\ln({t}^\delta_k) = i  \}    )\\
				&= \mu ( \{ x\in A_I:\ \arg\min_{k \in I}  c(x,u_k) -\ln(1-\delta)-\ln(t_k )= i  \}  )\\
				&=\mu(A_I) \mu_I (  \{ x\in A_I:\ \arg\min_{k \in I} c(x,u_k) -\ln(t_k )= i  \}  ) = \alpha_i.
				\end{align*} 
				Thus we have $\beta_i \le \alpha_i$.  Since $\sum_{i=1}^m \beta_i = 1 =\sum_{i=1}^m \alpha_i$, and   $\alpha_i \le \beta_i \le \alpha_i + \varepsilon/m$ when $i\notin I$, we see that 
				\[\sum_{i\in I} \alpha _i - \varepsilon \le  1 - \sum_{j\notin I} (\alpha_j + \varepsilon/m) \le  \sum_{i\in I} \beta_i \le \sum_{i\in I} \alpha _i\] 
				and we conclude $\beta_i > \alpha_i -\varepsilon$ for  $i\in I$.
								
				So far we have shown that for any $1\leq i \leq m$, we have
				\[\alpha_i - \varepsilon \leq \beta_i \leq \alpha_i + \varepsilon,\]
				which completes the proof of the first inclusion.
				
				We proceed to show the second inclusion $\subseteq$. Let  $\alpha\in H_{(u_i)}^\mu (t)$  for $t=t_I \times_I \{0_{I^c}\}$. By the definition of $H$ on the boundary of the simplex,  there exists a sequence
				\[(t_1^{(k)},\dots, t_m^{(k)})=t^{(k)} \to t =(t_1, \dots, t_m)\]
				with $t^{(k)} \in {\rm int} (\Delta_m)$, and
				\[H^\mu_{(u_i)_{i=1}^m}(t^{(k)}) \to \alpha.\] 
		In particular
				$\sum_{i\in I} t_i^{(k)}\to 1$ and $\sum_{j\notin I} t_j^{(k)}\to 0$.
				Note that for $x\in A_I$, as $k \to \infty$, the minimum in the definition will be attained (from some $k_0$ onwards) on an index  $i\in I$. Therefore 
				\[\sum_{j\in I} \alpha_j \leftarrow \sum_{j\in I} (H_{(u_i)_{i=1}^m}^\mu(t_i^{(k)}))_j = \mu ( x\in X:  \arg\min_{1\leq i \leq m}  c(x,u_i) -\ln(t_i^{(k)}) \in I ) \rightarrow \mu(A_I)\] (limits with respect to $k\to \infty$), which implies $\sum_{i\in I} \alpha_i = \mu (A_I)$ and thus $ \sum_{j \notin I} \alpha_j = \mu(A_I^c)$.

				Set ${t'}_i^{(k)} = \frac{t_i^{(k)}}{\sum_{j\in I} t_j^{(k)}}$ for $i\in I$. It is well defined (for large enough $k$, as $\sum_{j\in I} t_j =1$), and its limit is clearly ${t'}_i=\frac{t_i}{\sum_{j\in I} t_j}$. Hence, 
				\begin{align*}
				H^{\mu_I}_{(u_i)_{i\in I}}(t'_i)
				&= ( \mu_I( \{ x\in X:\ \arg\inf_{k \in I}   c(x,u_k)-\ln(t'_k) = i \} ))_{i\in I}\\
				&= ( \frac{1}{\mu(A_I)}\mu ( \{ x\in A_I:\ \arg\min_{k \in I} c(x,u_k) -\ln( t'_k )= i \} ) )_{i\in I}\\
				&= ( \frac{1}{\mu(A_I)}\mu (  \{ x\in X:\ \arg\min_{k \in I}   c(x,u_k) -\ln(t_k) = i  \}  )  )_{i\in I} 
				 =\frac{1}{\mu(A_I)}(\alpha_i)_{i \in I}.
				\end{align*}
				In the second to last step we used again the fact that the minimum can be attained at $i \in I$ only if $x \in A_I$.
				Thus,
				\[(\alpha_i)_{i\in I} \in \mu(A_I)H^{\mu_I}_{(u_i)_{i\in I}}\left(\frac{t_i}{\sum_{i\in I}t_i}\right) = \mu(A_I)H^{\mu_I}_{(u_i)_{i\in I}}\left((t_i\right)_{i\in I}).\]

				Setting ${t''}_i^{(k)} = \frac{t_i^{(k)}}{\sum_{j\notin I} {t}_j^{(k)}}$  for $i \notin I$, the sequence $({t''}_i^{(k)})_{k=1}^\infty$ has a converging subsequence in  $\Delta_{m-|I|}$, denote this subsequence by $({t''}_i^{(k_{l})})_{l=1}^\infty$, and its limit $({t''}_i)_{i\notin I}$. Once again, by the same argument as above, the image of this point under the map $H$ corresponding to $\mu_I^c$ is exactly
				\[H_{(u_i)_{i\notin I}}^{\mu_I^c}(t''_j) = \frac{1}{\mu(A_I^c)}(\alpha_j)_{j \notin I}.\] 
			\end{proof}

	\subsection{Transporting a non-degenerate measure to a discrete measure }
	We proceed to the proof of the following theorem, which is a version of Theorem \ref{thm:discrete-transport}, with an extra non-degeneracy assumption of $\mu$ (recall Definition \ref{def:non-degenerate-hall}).

\begin{thm}\label{thm:onto-under-goodness}
	Let $X$ and $Y$ be measure spaces,  $c:X\times Y \to (-\infty, \infty]$ measurable, fix $(u_i)_{i=1}^m\subset Y$ and let $\mu \in {\mathcal P}(X)$ be $c$-regular and supported on $\{x\in X: \exists i\in [m]\  c(x,u_i)<\infty\}$.  
	Assume, in addition, that $\mu$ is non-degenerate with respect to $(u_i)_{i=1}^m$. 
	Then, the mapping $H_{(u_i)_{i=1}^m}^\mu: {\rm int}(\Delta_m) \to  %
	(P ((u_i)_{i=1}^m, \mu))$ covers the set ${\rm int}
	(P ((u_i)_{i=1}^m, \mu))$, that is, for any $\alpha \in {\rm int} (P((u_i)_{i=1}^m, \mu)))$ there exists some $t\in{\rm int} (\Delta_m)$ such that 
	$H_{(u_i)_{i=1}^m}^\mu(t) = \alpha$. 
\end{thm}
			
			\begin{proof}
	Denote $H = H_{(u_i)_{i=1}^m}^\mu$ and $P = P ((u_i)_{i=1}^m, \mu)  $. By the non-degeneracy assumption, $P$ is full dimensional, and in particular has non-empty interior.
	If the image of $H $ did not 
				cover the interior of $P$, there would be some $\alpha \in {\rm int} P$ such that $H (t)\neq \alpha$ for all $t \in {\rm int} (\Delta_m)$. We use $\alpha$ to define the radial projection $p$ of $P\setminus\{\alpha\}$ to its boundary. It follows that $p \circ H $ is well defined and continuous. We then use the function $R$ as given in Proposition \ref{prop:mapping-P-to-simplex} to map the boundary of $P$ to the boundary of the simplex.  
				Since $H(\partial_I \Delta_m) \subset F_I$	we see that $R\circ p \circ H$ is a mapping from the simplex to its boundary which maps the $I^{th}$ facet to itself.

				Note that this composition map is a well defined function, i.e. a point-valued map: It is clearly point-valued on ${\rm{int}} (\Delta_m)$. Let $t \in \pa \Delta_m$, and take the minimal $I$ (with respect to inclusion) such that $t \in \pa_I \Delta_m$. Then, the coordinates $t_I$ are all non-zero, and by  Proposition \ref{prop:description-of-H} points in the set $H(t)$ differ only on their $I^c$ coordinates. Again by Proposition \ref{prop:description-of-H}, $H(t) \in F_I \subset \pa P$, so $p (H(t))=H(t)$. Further, since the map $R$ depends only on the $I$ coordinates of $H(t)$ (as $H(t)\in F_I$), we conclude that the set $H(t)$ is mapped to a single point, and thus $R\circ p \circ H$ is point-valued.
				
				Next we claim that the composition $R\circ p \circ H$ is a continuous function on $\Delta_m$. For points in the interior of $\pa \Delta_m$  this follows from the fact that all three maps are continuous (see Propositions \ref{prop:H-continuous} and \ref{prop:mapping-P-to-simplex}). We proceed to explain why the composition is continuous on the boundary.
				Let  $t = t_I \times_I 0_{I^{c}}$ be some boundary point, with $t_i>0$ for $i\in I$ (so,  $t\in {\rm relint} (\partial_I \Delta_m)$).
				Consider a sequence $t^{(k)} \to t$ with $(R\circ p \circ H)(t^{(k)})$ converging to some vector $s$ on the boundary of the simplex.   We need to show that $s = (R\circ p \circ H)(t)$. 
				By the definition of $H$ on boundary points, and the continuity of $p$ and $R$, we may without loss of generality assume $t^{(k)} \in {\rm int} (\Delta_m)$.  
				Indeed,  for any $t'\in \partial  \Delta_m$, $y\in H(t')$ and any $\eps>0$, there is some $t'_\eps \in {\rm int} (\Delta_m)$ with $\left|y - H(t'_\eps)\right|<\eps$, so given {\em any} sequence  $t^{(k)} \to t$  with $(R\circ p \circ H)(t^{(k)})$ converging to $s$ we can construct a sequence in the interior, converging to $t$, whose image under $R\circ p \circ H$  converges to the same $s$.  By definition of $H$, all accumulation points of the sequence $H(t^{(k)})$ belong to $H(t)$. By continuity of $R\circ p$, we conclude that
				all accumulation points of $R\circ p \circ H(t^{(k)})$
				(which we have assumed converge to the point $s$) belong to $(R\circ p)(H(t))$.  However, as we have already seen, $R\circ p \circ H(t)$ is a point, %
				and we get that  $s= R\circ p \circ H(t)$.

				However, there does not exist a continuous mapping from the simplex to its boundary which preserves the facets. Indeed, this can be shown, for example, using Brouwer's fixed point theorem -- as such a map  could then be composed with a permutation, arriving at a continuous mapping from the simplex to itself with no fixed point. Hence, $H$ covers the interior of $P$, and for every $\alpha\in {\rm int} (P)$ there is some preimage $t$. Moreover, this $t$ satisfies $t\in {\rm int} (\Delta_m)$, otherwise, if $t_i =0$ for some set of indices $i\in I$, then by Proposition \ref{prop:description-of-H}, $H(t)\in F_{I^c}$, which does not contain $\alpha$ (as $ F_{I^c}$ is not in the interior of $P$).
			\end{proof}			
		
			\subsection{Removing the non-degeneracy condition}\label{subsec:removing-non-degenerate}

		The only difference between Theorem \ref{thm:discrete-transport} and Theorem \ref{thm:onto-under-goodness}, apart from notation, is that in the latter we assume not only that  \[\{x\in X: c(x,u_i) <\infty \} \cap \{x\in X: c(x,u_j) <\infty \}\] 
		contains an open set for any $i\neq j$, but that $\mu$ is non-degenerate with respect to the vectors $(u_i)_{i=1}^m$, namely that 
			\[ \mu \left(\{x: c(x,u_i)<\infty\} \cap \{x: c(x,u_j)<\infty\}\right)>0.\]

			To remove this condition, we will use a straightforward  perturbation argument, similar to constructions used  for example, by McCann \cite{mccannthesis},   adding in this case uniform measures on small disks, and taking limits. More formally, make use of the following technical lemma.

\begin{lem}\label{fact:Pk-converges-to-P}
	Let $P=P((u_i)_{i=1}^m, \mu)$ be the Hall polytope associated with $(u_i)_{i=1}^m$ and the measure $\mu$ supported on $\{x\in X: \exists i\in [m]\  c(x,u_i)<\infty\}$, and assume
	 $P$ is full dimensional.
	Further assume that
	  for $1\le i<j\le m$ the intersection
				\[\{x\in X: c(x,u_i) <\infty \} \cap \{x\in X: c(x,u_j) <\infty \} \] 
				contains a disk, for any $i\neq j$, and let $\eta_{i,j}$ denote a uniform measure on this disk, with the constants chosen so that $\sum_{i<j\in [m]} \eta_{i,j}(X) = 1$.
	For any $k\in {\mathbb N}$ let 		 
			 $$\mu_k = \frac{1}{k} \sum_{i<j\in [m]} \eta_{i,j}  + (1-\frac{1}{k}) \mu,$$ and  $P_k=P((u_i)_{i=1}^m, \mu_k)$ the associated Hall polytope. Then,
				\begin{enumerate}[(i)]
					\item $H^{\mu_k}_{(u_i)_{i=1}^m} \to_{k\to \infty} H^{\mu}_{(u_i)_{i=1}^m}$ uniformly on $\Delta_m$, and
					\item $P_k \to P$ as $k\to \infty$ in the Hausdorff metric.
				\end{enumerate}
			\end{lem}

			 \begin{proof}

				Note first that each $\mu_k$ is non-degenerate, so by Proposition \ref{prop:Pfulldim} $P_k$ is full dimensional. Furthermore, we may apply				
				Theorem \ref{thm:onto-under-goodness}, and get that the mapping $H_k: = H^{\mu_k}_{(u_i)_{i=1}^m}: {\rm int}(\Delta_m) \to  P^{(k)} $ covers the set ${\rm int}(P^{(k)})$. Denote $H = H^{\mu}_{(u_i)_{i=1}^m}$.
				
				For (i), let $t=(t_1,\dots,t_m) \in {\rm int}(\Delta^m)$. 
				It will be convenient to recall the notation 
				$U_i = \left\{ x\in X:\ \arg\min_{1\leq j \leq m}   c(x,u_j)-\ln(t_j)  = i \right\}$. The $i^{th}$ component of the difference vector satisfies
				\begin{eqnarray*}
				\left|  \left(H_k(t) - H(t)\right)_i\right|
			 = \left| \mu_k\left(U_i\right)  - \mu\left( U_i \right)\right|
				= | \frac{1}{k} \sum_{i<j } \eta_{i,j}\left( U_i\right) 
		   - \frac{1}{k}\mu\left( U_i\right) | \le \frac{2}{k}.
				\end{eqnarray*}

				For (ii), let $\varepsilon >0$, we will show that for all $k>k_0=k(\varepsilon)$, $P_k \subset P +\varepsilon B_2^m$ and $P\subset P_k +\varepsilon B_2^m$. For the first inclusion, let $\alpha \in {\rm int} (P_k)$, then since $P_k$ is a Hall polytope of a non-degenerate measure, we apply Theorem \ref{thm:onto-under-goodness} and get a point $t\in \Delta_m$ for which $H_k(t)=\alpha$. As   $H(t)\in P$, the previous assertion (i) gives $\|\alpha - H(t)\|_2 \le \frac{2\sqrt{m}}{k}$, so ${\rm int}(P_k) \subset P + \frac{2\sqrt{m}}{k} B_2^m$ and therefore $P_k$ itself is also included in the same extension of $P$. For the second inclusion, let $\alpha \in P$ and let $\alpha^{(k)}$ be given by $\alpha_i^{(k)}:=\left(1-\frac{1}{k}\right)\alpha_i + \frac{1}{k}\sum_{\{j: i<j\}}\eta_{i,j}(X)$. We claim  $\alpha^{(k)}\in P_k$; we check that it satisfies all of the necessary inequalities. Clearly $\sum \alpha_i^{(k)}=1$, and as the support of $\eta_{i,j}$ is a subset of $A_{\{i\}}\cap A_{\{j\}}\subset A_I$ for all $i\in I$ we have
				\begin{align*}
				\sum_{i\in I} \alpha_i^{(k)}&=\left(1-\frac{1}{k}\right)\sum_{i\in I} \alpha_i +\frac{1}{k}\sum_{i\in I}\sum_{\{j:\, i<j\}}\eta_{i,j}(X) \le \left(1-\frac{1}{k}\right)\mu(A_I)+  \frac{1}{k}\sum_{i< j } \eta_{i,j}(A_I)=\mu_k(A_I).
				\end{align*}
				We compute 
				\[\|\alpha -\alpha^{(k)}\|_2= \left(\sum_{i=1}^m \frac{1}{k^2}(\alpha_i -\sum_{i\in I}\sum_{\{j: \,i<j\}}\eta_{i,j}(X))^2\right)^{1/2}\le \frac{2\sqrt{m}}{k}.\]
				Taking $k_0 = \frac{2\sqrt{m}}{\varepsilon}$ we see that both inclusions hold.
			\end{proof}

	We are now set up to prove the existence of a transport map  between strongly $c$-compatible measures, one of which is discrete and the other $c$-regular.

	\begin{proof}[Proof of Theorem \ref{thm:discrete-transport}]
		
			Let $\mu$ be a $c$-regular measure on $X$ and $\nu=\sum_{i=1}^m \alpha_i \ind_{u_i} $ a discrete measure on $Y$ which satisfy the assumptions of the theorem, and denote by $P=P(\mu,(u_i)_{i=1}^m)$ the  associated Hall polytope. 
			
		The condition of strong $c$-compatibility means precisely that for $I\neq \emptyset, [m]$ we have 
	$\sum \alpha_i < \mu(A_I)$, 
	or, in other words, that $\alpha = (\alpha_i)_{i=1}^m \in {\rm int} (P)$. In particular $P$ is non-empty and in fact full dimensional.  
			The conditions of Lemma \ref{fact:Pk-converges-to-P} are satisfied so we may use it to define $P_k$  and $H_k$ and find a  sequence of points $\alpha^{(k)}\in {\rm int}(P_k)$ such that $\alpha^{(k)}\to \alpha$ as $k\to \infty$. Take $t^{(k)}\in {\rm int}(\Delta_m)$ such that $H_k(t^{(k)})=\alpha^{(k)}$.  By the compactness of $\Delta_m$, there exists a converging subsequence of $t^{(k)}$, to some $t\in\Delta_m$, and we denote $t^{(k)} \to t$, abusing notation slightly. We claim that $t\in {\rm int}(\Delta_m)$ and that $H (t)=\alpha$. 
			Indeed,
			\begin{align*}
			\left|H (t) -\alpha \right| &\le \left|H (t) - H (t^{(k)}) \right|+  \left|H (t^{(k)}) - H_k(t^{(k)}) \right| + \left|H_k(t^{(k)})-\alpha^{(k)}\right|+ \left|\alpha^{(k)}-\alpha\right|.  
			\end{align*}
			Each of the terms tends to $0$ as $k\to \infty$: the leftmost by continuity of $H$ (Proposition \ref{prop:H-continuous}), the second by  uniform convergence of $H_k$ to $H$, the third term vanishes for every $k$ by choice of $t^{(k)}$, and the rightmost by choice of the sequence $\alpha^{(k)}$. Therefore $H(t) = \alpha$. Since $\alpha \in {\rm int} (P)$ we may use 
		  Proposition \ref{prop:description-of-H} to conclude that $t\in {\rm int} (\Delta_m)$. 
		  
		  We have thus established that $H$ is onto the interior of $P$. Recalling the construction in Lemma \ref{fact:c-subgradient-of-discrete-inf}, we have shown that the function 		
	 \[\varphi(x) = \min_{1\le i\le m} ( c(x,u_i)-\ln(t_i))\]
		satisfies that its $c$-subgradient supports a transport map from $\mu$ to $\nu$. 
		The function $\varphi$ is therefore our desired potential. Indeed, the map $T:X\to Y$  which maps the set $\{x:\ \min_{1\le j\le m} ( c(x,u_j)-\ln(t_j)) = i\}$ to $u_i$ for all $i \in [m]$, is a transport map (we define $T$ on the boundary of these sets arbitrarily, as it is $\mu$-negligible) and $(x,u_i) \in \partial^c\varphi$ for $x\in U_i$ by the first (and easy) part of Lemma \ref{lem:subgrad-of-basic-functions} from the appendix.
	\end{proof}

	\section{For the Polar cost: Maps versus Plans}\label{sec:plansvsmaps}
	
	Throughout the paper, we were careful to discuss transport plans, and not just maps. Indeed, even in the simplest cases of discrete measures, there is no reason for a transport {\em map} to exist, as it may require ``atom splitting'', a dangerous endeavor. Nevertheless, in the classical case, for example, when a transport plan from some absolutely continuous measure $\mu$ to a measure $\nu$ is concentrated on the usual subgradient of a convex function, $\varphi\in \cvx(\RR^n)$, it is easy to see that in fact one obtains a map, not just a plan. Indeed, a convex function has a unique subgradient almost everywhere. 
	
	For a general cost $c$ this is no longer the case, but for our main motivating example, the polar cost $p(x,y)=-\ln(\sp{x,y}-1)$, a similar argument works. Recall that for this cost the $p$-class is given by $-\ln(\v)$ where $\v\in \cvx _0(\R^n)$ is a geometric convex function, that is, a lower semi-continuous ~non-negative convex function with $\v(0)=0$. The $p$-subgradient of the function $-\ln(\varphi)$ coincides with the polar subgradient $\partial^\circ$, introduced in \cite{ArtsteinRubinstein}, of the function $\varphi\in \cvx_0(\R^n)$, and we have that
\begin{align*}
\partial^p(-\ln(\varphi))= \partial^\circ\varphi= \{(x,y): \varphi(x) \A \varphi (y) = \sp{x,y}-1>0\},
\end{align*}
where $\A\varphi(y)= \sup_{\{x:\, \sp{x,y}>1\}} \frac{\sp{x,y}-1}{\varphi(x)}$ is the polarity transform defined in \cite{ArtsteinMilmanHidden}. More details are provided in Appendix \ref{appendix:c-subg} together with the proof of the following lemma.

\begin{manuallemma}{\ref{lemma: polargrad}}
	Let $\varphi \in \cvx_0 (\RR^n)$ and let $x$ satisfy $\varphi(x) \in (0,\infty)$.  Then \begin{enumerate}[(i)]
		\item for any $z\in \partial \varphi (x)$ such that  $\sp{x,z} \neq \varphi(x)$, we have that 
		$y=\frac{z}{ \iprod{x}{z}  - \varphi(x)} \in \partial^\circ \varphi(x)$,
		\item for any $y \in \partial^\circ \varphi (x)$ there exists some $z\in \partial \varphi (x)$ such that  $\iprod{x}{z} \neq \varphi(x)$ and such that $y=\frac{z}{ \iprod{x}{z}  - \varphi(x)}$. 
	\end{enumerate}
\end{manuallemma}

	When $\varphi(x) = 0$ or $\varphi(x) = \infty$, then by definition, $\partial^\circ \varphi(x) = \emptyset$. When $\varphi(x) \in (0,\infty)$, the lemma implies that at a differentiability point of $\varphi$, the set $\partial^\circ\varphi(x)$ is either a singleton or is empty, which may happen only if the function $\varphi$ is linear on $[0,x]$. Our main theorem thus implies the following.

	\begin{manualtheorem}{\ref{thm:transpolar}}  
 	Let $X = Y = \RR^n$ and let $\mu,\, \nu\in {\mathcal P}(\RR^n)$ be probability measures with finite second moment, which are strongly $p$-compatible with respect to the polar cost, that is
 	\[\mu(K) + \nu (K^\circ) < 1 \]
 	for any convex set $K$ with $\mu(K)\neq 0,1$.
 	
 	Assume further that $\mu$ is absolutely continuous and that there exists some finite cost plan mapping $\mu$ to $\nu$. 
 	Then there exists $\varphi \in \cvx_0(\RR^n)$ such that $\grad \dual \varphi$  is an optimal transport map between $\mu$ and $\nu$, where 
 	\[ \grad \dual \varphi(x) = \{ y: \varphi(x) \A \varphi(y) = \iprod{x}{y}-1>0\}. \] 
 	In particular, for $\mu$-almost every $x$,  the set $\grad \dual \varphi(x)$ is a singleton. 
\end{manualtheorem}

	\begin{proof}
		By Theorem \ref{thm:transport-polar-comp}, we find a function $\varphi\in \cvx_0(\RR^n)$ such that there is an optimal plan $\pi$ concentrated on the graph of $\pa^p(-\log \varphi) = \pa^\circ \varphi$. We claim that $\mu$-almost everywhere, the set  $\pa^\circ \varphi(x)$ is a singleton, implying that $\pa^\circ \varphi$ is indeed a transport map. 
		
		Since $\pi$ is concentrated on $\pa^\circ \varphi$, the measure $\mu$ is concentrated on the projection of $\pa^\circ \varphi$, so in particular on the set of $x\in X$ with $\varphi(x)\neq 0, \infty$. We may also restrict to points in the interior of the domain of $\varphi$, as $\varphi$ is convex and points on the boundary of its domain have $\mu$-measure zero (using again that $\mu$ is absolutely continuous).
		We have that $\mu$-almost every point $x$ in the interior of the domain of $\varphi$ is a differentiability point of $\varphi$, and further that $\varphi$ does not vanish on $\mu$-almost every such point. Hence, by Lemma \ref{lemma: polargrad},  $\pa^\circ \varphi(x)$ is either a singleton or the empty set (in which case $x$ does not belong to the projection of $\pa^\circ \varphi$). We conclude that indeed $\pa^\circ \varphi(x)$ must be a singleton $\mu$-almost everywhere, as required.
	\end{proof}

\section{Decomposable pairs}\label{sec:decomposition}

We discussed in Section \ref{sec:c-compatibility} that when considering the transport problem of a measure $\mu\in {\mathcal P}(X)$ to $\nu \in {\mathcal P}(Y)$, with respect to a cost function $c:X\times Y\to (-\infty, \infty]$, where  $\mu$ and $\nu$ are $c$-compatible but not strongly $c$-compatible, the transport problem splits into two transport problems of disjointly supported measures. Let us make this observation more formal.

\begin{prop}\label{prop:decomposing}
Let $c:X\times Y\to (-\infty, \infty]$, and assume $\mu  \in \calP(X),\, \nu\in \calP(Y)$ are $c$-compatible measures which are not strongly $c$-compatible. There exists a $c$-class set $A\subset X$, and $B\subset Y$ such that $Y\setminus B$ is $c$-class, with $\mu(A) = \nu(B) \in (0,1)$,   such that, letting $\mu|_A$ and $\nu|_B$ denote the restricted measures, normalized, the pair $\mu|_A$ and $\nu|_B$ is $c$-compatible, as is the pair $\mu|_{X\setminus A}$ and $\nu|_{Y\setminus B}$. Moreover, 
any $\pi\in \Pi(\mu, \nu)$ which is concentrated in the set $S  =\{(x,y) : c(x,y) <\infty\}$, can be written as $\pi  = \mu(A) \pi_1 + (1-\mu(A)) \pi_2$, where $\pi_1 \in \Pi(\mu|_A, \nu|_B)$ and $\pi_2 \in \Pi(\mu|_{X\setminus A}, \nu|_{Y\setminus B})$, and $C(\mu, \nu) = \mu(A)C(\mu|_A, \nu|_B)+ (1-\mu(A))C(\mu|_{X\setminus A}, \nu|_{Y\setminus B})$. 
\end{prop}

\begin{proof}
Indeed, by	Lemma \ref{lem:polar-compatible1}, the fact that the measures are not strongly $c$-compatible  implies that there exists some set $A\subset X$, which is a $c$-class set (this means there is some $D\subset Y$, which can also be assumed to be a $c$-class set, such that $A = \{x: \forall y\in D,\, c(x,y) = \infty\}$), and such that 
$\mu(A) \in (0,1)$ and 
\[ \mu (A) + \nu(\{y: \forall x\in A,\,\, c(x,y) = \infty\}) = 1. \] 
	Rearranging, this means that 
	\[ \mu (A)  = \nu(\{y: \exists x\in A,\,\, c(x,y) <\infty\})\quad{\rm and}\quad \mu(X\setminus A) = \nu(\{y: \forall x\in A ,\,\, c(x,y) = \infty\}). \] 
	Let $B = \{y: \exists x\in A,\,\, c(x,y) <\infty\} = Y\setminus D$. 
To see that $\mu|_A$ and $\nu|_B$ are $c$-compatible, 	letting $\mu(A) = a$, say, and fixing some set $A'\subset A$, we see that 
\begin{eqnarray*}
	&& \mu(A') + \nu(\{y\in B: \forall x\in A' ,\,\,c(x,y) = \infty\})  = \\
	&& \mu(A') + \nu(\{y\in Y: \forall x\in A',\,\, c(x,y)  =  \infty\}) -\nu(Y\setminus B)  \le  1-(1-a) = a,
\end{eqnarray*}
as required. Similarly for the complementary measures. 
If a transport plan $\pi \in \Pi(\mu, \nu)$ is concentrated on $S$, then $\pi$ cannot have non-zero measure in $A\times (Y\setminus B)$ or in $(X\setminus A)\times B$. 
Indeed, as 
	\[\mu(A) = \pi((A\times Y)\cap S)= \pi(A\times B)\le \nu(B) = \mu(A)\]
implying that we have equalities all along, and $\pi (A\times (Y\setminus B)) = 0$. Similarly, as $Y\setminus B = D$ is a $c$-class set, $D = \{ y: \forall x\in A,\, c(x,y) = \infty \}$, so $D$ must be mapped to $X\setminus A$, and as these sets have the same measure, 
\[ \nu(D) = \pi((X\times D) \cap S) = \pi ((X\setminus A) \times D) \le \mu(X\setminus A) = \nu(D),  \]
and by the same reasoning,  $\pi ((A\setminus X)\times B) = 0$. 
(Figure \ref{fig:example35} is a good illustration of this event.)

In other words, such a  	
 transport plan can be split into its components, $\pi_1 = \pi|_{A\times B}$ and $\pi_2 = \pi|_{(X\setminus A) \times (Y\setminus B)}$ (where as above, restriction means to restrict, and renormalize to a probability measure). 
  This completes the proof.  
\end{proof}

Of course, the fact that the problem splits into two sub-problems does not necessarily imply we may solve it in a satisfactory way. Indeed, it may be the case that each sub-problem has an associated potential function, but these two functions cannot be ``glued'' so as to form a potential for the original problem. This is the case for example for the polar cost in the following example

\begin{exm}\label{ex:no potential}
Consider the set \[
A= \{ (x,y): \ x \in (\tfrac{1}{2},1), \ y=3-2x \} \cup \{ (x,y): \ x \in (1,2), \ y=\tfrac{3}{2}-\tfrac{1}{2}x \} \subset \R^+ \times \R^+.\]
The set is a $p$-cyclically monotone (with respect to the polar cost $p(x,y) = -\ln(\iprod{x}{y}-1)$   since for every point $(x,y)\in A$ we have $\sp{x,y}>1$ and it is a graph of non-increasing function on its domain, which characterized $p$-cyclically monotone sets on the ray $\RR^+$, see \cite{paper1}. 
However, the set is not $p$-path-bounded, and thus admits no potential. 

Next, consider the measure $\mu = \nu$ on $[1/2, 2]$ with density $1$ on $[1/2, 1]$ and density $1/2$ on $[1,2]$. This is a probability measure. In fact, $\mu$ and $\nu$ are $p$-compatible as the normalized uniform measure on the set $A$ constitutes a plan $\pi \in \Pi(\mu, \nu)$. However, they are not strongly $p$-compatible since the set $A = [1/2, 1]$ must be mapped to $B = [1, 2]$ and vice versa. 

In this case we see the splitting very clearly, and indeed $A$ is written as the union of two sets, each of which admits a potential (so, in particular, each is $c$-path-bounded, and is an optimal plan between the corresponding restricted measures). However, there is no potential for the full set $A$, as it is not $c$-path-bounded, and in particular no ``gluing'' of the two potentials is possible. 
 \end{exm}

\appendix
\addcontentsline{toc}{section}{Appendices}

\section{$c$-subgradients and polar subgradients}\label{appendix:c-subg}

Since $c$-subgradients play such an important role in this theory, we gather here some relevant information regarding them but which we did not include in the main text so as not to disturb its flow.

Let us recall that given a function $\varphi$ in the $c$-class, its $c$-subgradient is defined by 
\begin{equation*}
	\pa^c\varphi =\{(x,y):\, \varphi(x)+\varphi^c(y)=c(x,y) \,\text{ and } \, c(x,y)<\infty\}.
\end{equation*}

Denoting by $\pa^c\varphi(x)$ the set of points $y\in Y$ for which $(x,y)\in \pa^c \varphi$,  we have by definition that 
\[x\in \pa^c\varphi^c(y)\iff y\in \pa^c\varphi(x).\]
Notice that $y\in \pa^c\varphi(x)$ if and only if the function $c(\cdot,y)-\varphi^c(y)$ is above $\varphi$ and coincides with it at $x$. This provides the first simple but useful way to think about $c$-subgradients, summarized in Lemma \ref{lem:c-sub-geom}. 
Given a function $\varphi$ in the $c$-class, it is the image, under the $c$-transform, of another $c$-class function $\psi = \varphi^c$ and therefore, it can be written as an infimum over basic functions as follows:
\[ \varphi (x) = \inf_y \left(c(x,y) - \varphi^c(y)\right).\]
All the functions on the right hand side lie above $\varphi$. 
If any one of the basic functions (indexed by $y$)  on the right hand side is tangent to $\varphi$ at the point $x$, then the pair $(x,y)$ belongs to $\partial^c \varphi$, and $y \in \partial^c \varphi (x)$.  In other words 

\begin{lem}\label{lem:c-sub-geom}
	Let $\varphi$ be a $c$-class function, and $x\in X$ and assume that $\v(x)<\infty$. Then
	$y_0\in \partial ^c \v(x)$ if and only if  $c(x,y_0)<\infty$ and 
	the function $\ell(z)=c(z,y_0)-c(x,y_0)+\varphi(x)$ satisfies \[\ell(z)\ge \varphi(z) \ \text{ for all } \ z\in X.\]	
\end{lem}

\begin{proof}
	By the definition we have that $y_0\in \partial^c \varphi (x)$ if and only if \[\varphi(x)+ \varphi^c(y_0)=c(x,y_0)<\infty.\] Using the definition of the $c$-transform we see that
	\[\varphi(x) = c(x,y_0)-\varphi^c (y_0)= \sup_z(c(x,y_0)-c(z,y_0)+\varphi(z)),\]
	which holds if and only if for all $z$ we have $c(z,y_0)-c(x,y_0)+\varphi(x) \ge \varphi(z)$.
\end{proof}

It is useful to understand the structure of the $c$-subgradient of the basic functions. In parallel to the classical case, where the linear functions have constant subgradient, we show that under mild assumptions the same is true for $c$-subgradients of basic functions. This was, of course, our motivation for using the specific candidates for the potential functions in Section \ref{sec:discrete-pf}.

\begin{lem}\label{lem:subgrad-of-basic-functions}
	Let $X,\,Y$ be  measure spaces and let  $c:X\times Y \to (-\infty, \infty]$ be a measurable cost function. Consider a basic function $\varphi(x) = c(x,y_0)+ t$ for some $y_0 \in Y$. If $c(x,y_0)<\infty$, 
	then $y_0\in \cgrad \v(x)$. If, in addition, for any $y_1\neq y_0$ we have that $\inf_z\left( c(z,y_1) - c(z,y_0)\right)$ is not attained at $x$ (for example, if the infimum is $-\infty$, or bounded but not attained at all) then $\{y_0 \}= \cgrad \v(x)$. 
\end{lem}

\begin{proof}
	Indeed, let $\v$ be as in the statement. From the definition it follows that $y\in \cgrad \v(x)$ if and only if $c(x,y)<\infty$ and
	\[c(x,y)-\v(x)=\v^c(y) =\inf_{z} ( c(z,y)-\v(z)), \]
	which can be reformulated as
	\[c(x,y)-\v(x) \le c(z,y_1)-\v(z) \ \ \ \text{ for all } z\in X.\]
	Plugging in the definition of $\v$ we get
	\[c(x,y)-c(x,y_0)\le c(z,y)-c(z,y_0) \ \ \ \text{ for all } z\in X.
	\]
	We see that $y=y_0$ always satisfies the equality, so that $y_0\in \partial^c \v(x)$.
	Clearly for $y_1\neq y_0$, such an inequality means precisely that the infimum is attained at $x$. 
\end{proof}

An important and motivating first example is the one coming from the clasical cost function $c(x,y) = -\iprod{x}{y}$. 
 
 \begin{exm}
 	For the cost function $c(x,y)=-\iprod{x}{y}$,  
 	whose transport plans and maps coincide with those associated to the quadratic cost, the $c$-subgradient coincides, up to a minus sign,  with the well known subgradient. More formally, a function $\varphi$ is in the $c$-class if and only if $-\varphi\in \cvx(\RR^n)$, namely is convex and lower semi-continuous. Denoting $\psi = -\varphi$ and 
 	 using the definition of the $c$-transform we see that $(x,y)\in \partial ^c \varphi$ if and only if  for all $z\in X$ we have
 	\[\psi(z) - \psi(x) = \v(x)-\v(z)\ge  c(x,y)-c(z,y).\]
 	Plugging in the quadratic cost we indeed get that $y\in \partial^c \v(x)$ if for all $z$ it holds that $ \psi(x) + \iprod{z-x}{y} \le \psi(z)$, namely $y\in \partial \psi(x)$.

 \end{exm} 

The second motivating example, which is our main point of interest, is that of the polar cost $p:\RR^n \times \RR^n \to (-\infty, \infty]$, which we once again recall 
\begin{equation*}
p(x,y) =- \ln(\langle x,y \rangle -1)_+ = \begin{cases}
- \ln(\langle x,y \rangle -1), & \ \ \text{if } \langle x,y \rangle >1 \\
+\infty, & \ \ \text{otherwise.}
\end{cases} 
\end{equation*}  

It was shown in  \cite{hila} that
for the polar cost the $p$-class consists of all functions of the form $-\ln(\varphi)$, where $\varphi$ is a geometric convex function, that is, a lower semi-continuous ~non-negative convex function with $\v(0)=0$. The associated cost transform is linked with the $\A$-transform defined in \cite{ArtsteinMilmanHidden} and given by
\begin{equation}\label{def:A-transform}
\A\varphi(y)= \sup_{\{x:\, \sp{x,y}>1\}} \frac{\sp{x,y}-1}{\varphi(x)}.
\end{equation}

More precisely, one may easily verify that $- \ln(\A \v)=(-\ln(\v))^p$. Further, the $p$-subgradient of the function $-\ln(\varphi)$ can be rewritten as the polar subgradient $\partial^\circ$, introduced in \cite{ArtsteinRubinstein}, of the function $\varphi\in \cvx_0(\R^n)$. Indeed, we have that
\begin{align}\label{def:polar-subgrad}
\partial^p(-\ln(\varphi))= \partial^\circ\varphi= \{(x,y): \varphi(x) \A \varphi (y) = \sp{x,y}-1>0\}.
\end{align}
 This convenient form is a reason for us to sometimes consider a ``multiplicative'' setting, where the basic functions are of the form
\[ \varphi_{u,t}(x) =  t (\iprod{x}{u }-1)_+.\]

The next lemma, which is a version of \cite[Lemma 3.3]{ArtsteinRubinstein}, describes the connection between the polar subgradient and the classical subgradient. We will use the following notation for the zero set $Z_\varphi = \{ x: \varphi(x) = 0\}$ and ${\rm dom}( \varphi)  = \{x: \varphi (x) <\infty\}$ for the domain where $\varphi$ is finite.

\begin{lem}\label{lemma: polargrad} 
Let $\varphi \in \cvx_0 (\RR^n)$ and let $x\in { \rm dom}(\varphi)\setminus Z_\varphi$. Then \begin{enumerate}[(i)]
    \item for any $z\in \partial \varphi (x)$ such that  $\sp{x,z} \neq \varphi(x)$, we have that 
	$y=\frac{z}{ \iprod{x}{z}  - \varphi(x)} \in \partial^\circ \varphi(x)$,
	\item for any $y \in \partial^\circ \varphi (x)$ there exists some $z\in \partial \varphi (x)$ such that  $\iprod{x}{z} \neq \varphi(x)$ and such that $y=\frac{z}{ \iprod{x}{z}  - \varphi(x)}$. 
\end{enumerate}
\end{lem}

 \begin{proof} 
 \textit{(i)} Let $z\in \partial \varphi (x)$ with  $\sp{x,z} \neq \varphi(x)$, which means that for every $w$ we have $\iprod{w}{z}-\varphi(w)\le \iprod{x}{z}-\varphi(x)$. In particular, %
 $\sp{x,z}-\v(x)> 0$. Hence, letting $y=\frac{z}{ \iprod{x}{z}  - \varphi(x)}$ we have that $\sp{x,y}>1$.
 
 To show that $y \in \partial^\circ \varphi(x)$, it remains to show that $ \varphi(x)\A  \varphi (y)  = \iprod{x}{y}-1$. According to the definition of $\A$, this holds if for every  $w$ with $\iprod{w}{y}>1$ and $\varphi(w)>0$, we have 
 \[  \frac{\iprod{w}{y}-1}{\varphi(w)}\le \frac{\iprod{x}{y}-1}{\varphi(x)}. \] 
 Plugging in $y$ and rearranging gives 
 \[  \frac{\iprod{w}{\frac{z}{ \iprod{x}{z}  - \varphi(x)}}-1}{\varphi(w)}\le \frac{\iprod{x}{\frac{z}{ \iprod{x}{z}  - \varphi(x)}}-1}{\varphi(x)} =  {\frac{1}{ \iprod{x}{z}  - \varphi(x)}}. \]
Using that   $\iprod{x}{z}-\varphi(x)> 0$, the above inequality is equivalent to our initial assumption $\iprod{w}{z}-\varphi(w)\le \iprod{x}{z}-\varphi(x)$.

 	\noindent \textit{(ii)} Given $y \in \partial^\circ \varphi (x)$ it follows from the definition that $\iprod{x}{y}>1$. Consider 
 	\[z = \frac{ y \varphi(x)}{\iprod{y}{x}-1},\] 
 	which is well defined, and also implies that 
 	$y=\frac{z}{ \iprod{x}{z}  - \varphi(x)}$. We need to show that $z\in \partial \varphi(x)$ and  $\iprod{z}{x}\neq \varphi(x)$.

 	The latter follows easily since $\sp{z,x}=\v(x)\big(1+\tfrac{1}{\sp{x,y}-1}\big) $ and once again that $\sp{x,y}>1$. For the former, we use as before that if $y\in \partial^\circ \varphi (x)$ then for any $w$ with $\iprod{w}{y}>1$ and $\varphi(w) >0$ we have $ \frac{\iprod{w}{y}-1}{\varphi(w)}\le \frac{\iprod{x}{y}-1}{\varphi(x)}$.  
Plugging in $y$ and rearranging, we get that \[\varphi(x) + \iprod{w-x}{z} \le \varphi(w)\]
holds for any $w$ such that $\iprod{w}{z}>\iprod{x}{z}-\varphi(x)$ and $\varphi(w)>0$. In the case when $w$ is such that 
 	$\iprod{w}{z}\le \iprod{x}{z}-\varphi(x)$, this actually means that $	\varphi(x) + \iprod{w-x}{z} \le 0$ 
 and since the geometric convex functions are non-negative the desired inequality trivially follows.
 
 It remains to consider the case when $w\in Z_\v$, i.e. when $\varphi(w)=0$. Then, plugging in the previously defined $z$, we have that the inequality defining the subgradient of $\v$ at $x$ becomes simply
 \[ \sp{w,y} \le 1.\]
 
  That is, we need to show that $\partial^\circ \varphi (x)$
is contained in the polar set of $Z_\varphi$. Indeed, $y\in \partial^\circ \varphi(x)$ implies in particular that $y \in \dom(\A\varphi)$ (since the value of $\A \varphi(y) = \frac{\iprod{x}{y}-1}{\varphi(x)} <\infty$) and it follows from the definition of $\A$ that $\dom (\A\varphi)\subset Z_\varphi^\circ$, which completes the proof. %
 \end{proof}

 We end this appendix with one explicit example of 
 a function and its $p$-subgradient. More examples and applications can be found in \cite{hila,kasia-thesis} and in the forthcoming \cite{paper3}.
 
\begin{exm}
	Let $\varphi(x) = |x|^2/2$, in which case ${\A}\varphi(y)= |y|^2/2$ and the supremum in the definition of $\A \varphi$ is satisfied for $x = 2y/|y|^2$.  Hence,  $\partial^\circ \varphi(x) = \frac{2x}{|x|^2}$.
	Note that the mapping $x\mapsto \partial^\circ \varphi(x)$ in this case is a (rescaled) spherical inversion. 
\end{exm}

\bibliographystyle{amsplain}
\addcontentsline{toc}{section}{References}\bibliography{Constructing-potentials}

\smallskip \noindent
School of Mathematical Sciences, Tel Aviv University, Tel Aviv 69978, Israel  
\smallskip \noindent

{\it e-mail}: shiri@tauex.tau.ac.il\\
{\it e-mail}: shaysadovsky@mail.tau.ac.il\\
{\it e-mail}: kasiawycz@outlook.com\\

\end{document}